\documentclass[11pt,english]{amsart}

\usepackage{amsmath,amsfonts,amssymb,amsthm,amscd}
\usepackage{dsfont, mathrsfs,enumerate,eucal,fontenc}
\usepackage[varg]{txfonts}
\usepackage{xspace}
\usepackage[english]{babel}

\usepackage{xy, xypic}
\usepackage{color} 
\usepackage{graphicx}
\usepackage{psfrag}  
\usepackage{pst-all}
\usepackage{youngtab}

\usepackage{xcolor}
\definecolor{rouge}{rgb}{0.85,0.1,.4}
\definecolor{bleu}{rgb}{0.1,0.2,0.9}
\definecolor{violet}{rgb}{0.7,0,0.8}
\usepackage[colorlinks=true,linkcolor=bleu,urlcolor=violet,citecolor=violet]{hyperref}
\usepackage{breakurl}
\usepackage{xspace}

\theoremstyle{plain}
\newtheorem{theorem}{Theorem}[section]
\newtheorem{lemma}[theorem]{Lemma}
\newtheorem{coro}[theorem]{Corollary}
\newtheorem{prop}[theorem]{Proposition}

\newtheorem{theo}[theorem]{Theorem}
\theoremstyle{definition}
\newtheorem{defi}[theorem]{Definition}
\theoremstyle{remark}

\newtheorem{claim}[theorem]{Claim}

\newtheorem{quest_alpha}{Question}
\def\quest_alpha{\Alph{quest_alpha}}

\def\g{{\mathfrak{g}}}

\def\k{{\Bbbk}}

\def\rg{\ell}

\def\leq{\leqslant}

\def\N{\mathbb{N}}

\def\no{n$^{\circ}$}

\def\poie#1#2#3#4#5#6#7#8#9{\def\un{#5#6#7#8#9}\def\deux{#6#7#8#9}\def\trois{#2#4#8#9}
\def\quatre{#8#9}\def\cinq{#5#6#7}\def\six{#6#7}\def\sept{#2#4}
\ifx\un\empty {#1}_{#2}{#3 \hskip 0.15em}{#1}_{#4} \else \ifx\deux\empty 
{#5}(#1_{#2}){#3 \hskip 0.15em}{#5}(#1_{#4})
\else \ifx\trois\empty {#5}_{#6}(#1){#3 \hskip 0.15em}{#5}_{#7}(#1) 
\else \ifx\quatre\empty {#5}_{#6}(#1_#2){#3 \hskip 0.15em}{#5}_{#7}(#1_#4) 
\else \ifx\cinq\empty {#1}_{#2}^{#8}{#3 \hskip 0.15em}#1_#4^{#9} 
\else \ifx\six\empty {#5}(#1_{#2}^{#8}){#3 \hskip 0.15em}{#5}(#1_{#4}^{#9}) 
\else \ifx\sept\empty {#5}_{#6}^{#8}(#1){#3 \hskip 0.15em}{#5}_{#7}^{#9}(#1) \else
{#5}_{#6}(#1_{#2}^{#8})^{#9}{#3 \hskip 0.15em}{#5}_{#7}(#1_{#4}^{#8})^{#9} 
\fi \fi \fi \fi \fi \fi \fi}
\def\poi#1#2#3#4#5#6#7{\def\un{#5#6#7}\def\deux{#6#7}
\def\trois{#2#4} \def\cinq{#3#4#5}
\ifx\un\empty {#1}_{#2}{#3 \hskip 0.15em}{#1}_{#4} \else
\ifx\deux\empty {#5}(#1_{#2}){#3 \hskip 0.15em}{#5}(#1_{#4}) \else
\ifx\trois\empty {#5}_{#6}(#1){#3 \hskip 0.15em}{#5}_{#7}(#1) \else
{#5_{#6}}(#1_{#2}){#3 \hskip 0.15em}{#5_{#7}}(#1_{#4}) \fi \fi \fi}
\def\rond{\raisebox{.3mm}{\scriptsize$\circ$}}

\def\tens{\raisebox{.3mm}{\scriptsize$\otimes$}}

\def\dv#1#2{\langle {#1}\,,{#2}\rangle}

\def\tk#1#2{{#2}\otimes _{#1}}

\def\ec#1#2#3#4#5{\def\un{#3#4#5}\def\deux{#3#5}\def\trois{#3}
\def\four{#2#4#5}\def\five{#2#5}\def\six{#2}\def\seven{#3#4}
\def\eight{#2#4} \def\nine{#2#3#4}
\ifx\nine\empty {\rm #1}_{#5} \else
\ifx\un\empty {\rm #1}({\goth #2}) \else
\ifx\deux\empty {\rm #1}({\goth #2}_{#4}) \else
\ifx\trois\empty {\rm #1}_{#5}({\goth #2}_{#4}) \else
\ifx\four\empty {\rm #1}(#3) \else
\ifx\five\empty {\rm #1}(#3_{#4}) \else
\ifx\six\empty {\rm #1}_{#5}(#3_{#4}) \else
\ifx\seven\empty {\rm #1}_{#5} ({\goth#2})\else
\ifx\eight\empty {\rm #1}_{#5}({#3})
\fi \fi \fi \fi \fi \fi \fi \fi \fi}
\def\hec#1#2#3#4#5{\def\un{#3#4#5}\def\deux{#3#5}\def\trois{#3}
\def\four{#2#4#5}\def\five{#2#5}\def\six{#2}\def\seven{#3#4}
\def\eight{#2#4} \def\nine{#2#3#4}
\ifx\nine\empty \hat{{\rm #1}}_{#5} \else
\ifx\un\empty \hat{{\rm #1}}({\goth #2}) \else
\ifx\deux\empty \hat{{\rm #1}}({\goth #2}_{#4}) \else
\ifx\trois\empty \hat{{\rm #1}}_{#5}({\goth #2}_{#4}) \else
\ifx\four\empty \hat{{\rm #1}}(#3) \else
\ifx\five\empty \hat{{\rm #1}}(#3_{#4}) \else
\ifx\six\empty \hat{{\rm #1}}_{#5}(#3_{#4}) \else
\ifx\seven\empty \hat{{\rm #1}}_{#5} ({\goth#2})  \else
\ifx\eight\empty \hat{{\rm #1}}_{#5}({#3})
\fi \fi \fi \fi \fi \fi \fi \fi \fi}
\def\e#1#2{\ec {#1}#2{}{}{}}
\def\es#1#2{\ec {#1}{}{#2}{}{}}

\def\ai#1#2#3{\def\deux{#2#3} \def\trois{#3} \def\quatre{#2} 
\ifx\deux\empty \es S{{\goth #1}}^{{\goth #1}} \else
\ifx\trois\empty \es S{{\goth #1}^{#2}}^{{\goth #1}^{#2}} \else
\ifx\quatre\empty \es S{{\goth #1}_{#3}}^{{\goth #1}_{#3}} \else
\es S{{\goth #1}_{#3}^{#2}}^{{\goth #1}_{#3}^{#2}} \fi \fi \fi}
\def\aii#1#2#3#4{\def\deux{#2#3} \def\trois{#3} \def\quatre{#2} 
\ifx\deux\empty \sy {#4}{{\goth #1}}^{{\goth #1}} \else
\ifx\trois\empty \sy {#4}{{\goth #1}^{#2}}^{{\goth #1}^{#2}} \else
\ifx\quatre\empty \sy {#4}{{\goth #1}_{#3}}^{{\goth #1}_{#3}} \else
\sy {#4}{{\goth #1}_{#3}^{#2}}^{{\goth #1}_{#3}^{#2}} \fi \fi \fi}

\def\Bbb{\mathbb}
\def\goth{\mathfrak}
\def\cal{\mathcal}

\def\gi#1#2#3#4{\def\trois{#3#4} \def\quatre{#4}\def\cinq{#3}\ifx\trois\empty {\rm i}_{#1,{\goth #2}}
\else \ifx\quatre\empty {\rm i}_{#1_{#3},{\goth #2}} \else\ifx\cinq\empty {\rm i}_{#1,{\goth #2}_{#4}} \else {\rm i}_{#1_{#3},{\goth #2}_{#4}} \fi \fi \fi}
\def\j#1#2{\def\deux{#2} \ifx\deux\empty {\rm rk}\hskip .125em{{\goth #1}} \else {\rm rk}\hskip .125em{{\goth #1}_{#2}} \fi}
\def\aj#1#2{\def\deux{#2} \ifx\deux\empty {\rm j}_{{\goth #1}} \else {\rm j}_{{\goth #1}_{#2}} \fi}
\def\an#1#2{\def\deux{#2} \ifx\deux\empty {\cal O}_{#1} \else {\cal O}_{#1,#2} \fi }
\def\han#1#2{\def\deux{#2} \ifx\deux\empty {\hat{{\cal O}}}_{#1} \else {\hat{{\cal O}}}_{#1,#2} \fi }

\def\dim{{\rm dim}\hskip .125em}
\def\dd{{\rm d}}
\def\ad{{\rm ad}\hskip .1em}
   
\def\tr{{\rm tr}\hskip .125em}

\def\deg{{\rm deg}\hskip .125em}
\def\rk{{\rm rk}\hskip .125em}

\def\gm{{\rm G}_{\rm m}}

\def\sy#1#2{{\rm S}^{#1}(#2)}

\def\ie#1{\hskip .125em ^{e}\hskip -.125em{#1}}

\begin{document}

\title[The symmetric invariants]
{The symmetric invariants of centralizers and Slodowy grading II}

\author[Jean-Yves Charbonnel]{Jean-Yves Charbonnel}
\address{Jean-Yves Charbonnel, Universit\'e Paris Diderot - CNRS \\
Institut de Math\'ematiques de Jussieu - Paris Rive Gauche\\
UMR 7586 \\ Groupes, repr\'esentations et g\'eom\'etrie \\
B\^atiment Sophie Germain \\ Case 7012 \\ 
75205 Paris Cedex 13, France}
\email{jean-yves.charbonnel@imj-prg.fr}

\author[Anne Moreau]{Anne Moreau}
\address{Anne Moreau, Laboratoire de Math\'ematiques et Applications\\
T\'el\'eport 2 - BP 30179\\
Boulevard Marie et Pierre Curie\\
86962 Futuroscope Chasseneuil Cedex, France}
\email{anne.moreau@math.univ-poitiers.fr}

\subjclass
{17B35,17B20,13A50,14L24}

\keywords{symmetric invariant, centralizer, polynomial algebra, Slodowy grading}

\date\today

\begin{abstract}
Let $\g$ be a finite-dimensional simple Lie algebra of rank $\rg$ over an 
algebraically closed field $\k$ of characteristic zero, and let 
$(e,h,f)$ be an $\mathfrak{sl}_2$-triple of $\g$. 
Denote by $\g^{e}$ the centralizer of $e$ in $\g$ 
and by $\ai g{e}{}$ the algebra of symmetric invariants of $\g^{e}$. 
We say that $e$ is good if the nullvariety of some $\rg$ homogeneous elements of 
$\ai g{e}{}$ in $({\goth g}^{e})^{*}$ has codimension $\rg$. If $e$ is good then 
$\ai g{e}{}$ is a polynomial algebra. 
In this paper, we prove that the converse of the  main result of \cite{CM1} 
is true. Namely, we prove that $e$ is good if and only if for some
homogeneous generating sequence $\poi q1{,\ldots,}{\rg}{}{}{}$ of $\ai g{}{}$, 
the initial homogeneous components of their restrictions to $e+\g^{f}$ are algebraically 
independent over~$\k$. 
\end{abstract}

\maketitle

\tableofcontents

\section{Introduction} \label{i}

\subsection{}\label{i1}
Let $\g$ be a finite-dimensional simple Lie algebra of rank $\rg$ over an 
algebraically closed field $\k$ of characteristic zero, let $\dv ..$ be the Killing form 
of $\g$ and let $G$ be the adjoint group of $\g$. If ${\goth a}$ is a subalgebra of $\g$,
we denote by S$({\goth a})$ the symmetric algebra of ${\goth a}$. For $x\in\g$, we 
denote by $\g^{x}$ the centralizer of $x$ in $\g$ and by $G^{x}$ the stabilizer of $x$ 
in $G$. Then 
${\rm Lie}(G^{x})={\rm Lie}(G^{x}_{0})=\g^x$ where $G^{x}_{0}$ is the 
identity component of $G^{x}$.  Moreover, $\es S{\g^{x}}$ is a $\g^x$-module and 
$\ai gx{}=\es S{\g^{x}}^{G^{x}_{0}}$. 

In \cite{CM1}, we continued the works of \cite{PPY} 
and we studied the question on whether the algebra 
$\es S{\g^{x}}^{\g^{x}}$ is polynomial in $\rg$ variables; 
see \cite{Y3,CM,JS,Y4} for other references related to the topic. 

\subsection{}
Let us first summarize the main results of \cite{CM1}. 
\begin{defi}[{\cite[Definition 1.3]{CM1}}] \label{di}
An element $x\in{\goth g}$ is called a {\em good element of ${\goth g}$} if for some 
homogeneous sequence $(\poi p1{,\ldots,}{\rg}{}{}{})$ in $\ai g{x}{}$, the nullvariety of 
$\poi p1{,\ldots,}{\rg}{}{}{}$ in $({\goth g}^{x})^{*}$ has codimension $\rg$ in 
$({\goth g}^{x})^{*}$. 
\end{defi}
Thus an element $x\in{\goth g}$ is good if the nullcone of $\es S{\g^{x}}$,  
that is, the nullvariety in $({\goth g}^{x})^{*}$ of the augmentation ideal 
$\ai g{x}{}_+$ of $\ai g{x}{}$, is a complete intersection in $({\goth g}^{x})^{*}$ 
since the transcendence degree over $\k$ of the fraction 
field of $\ai g{x}{}$ is~$\ell$ by the main result of \cite{CM}.  

For example, regular nilpotent elements are good; see the introduction of \cite{CM1} 
for more details and other examples. 

\begin{theorem}[{\cite[Theorem 3.3]{CM1}}] \label{ti3}
Let $x$ be a good element of ${\goth g}$. Then $\ai gx{}$ is a polynomial algebra and 
$\es S{{\goth g}^{x}}$ is a free extension of $\ai gx{}$. 
\end{theorem}

An element $x$ is good if and only if so is its nilpotent 
component in the Jordan decomposition \cite[Proposition 3.5]{CM1}.  
As a consequence, we can restrict the study to the case of nilpotent elements. 

Let $e$ be a nilpotent element of $\g$. 
By the Jacobson-Morosov Theorem, $e$ is embedded into an 
${\goth {sl}}_2$-triple 
$(e,h,f)$ of $\g$. 
Identify $\g$ with $\g^*$, and 
${\goth g}^{f}$ with $({\g}^{e})^*$, through the Killing isomorphism 
$\g \to \g^*, \, x \mapsto \dv{x}{.}$. 
Thus we have the following algebra 
isomorphisms: $\e Sg \simeq\k[\g^*]\simeq\k[\g]$ 
and $\es S{{\goth g}^{e}}\simeq\k[(\g^{e})^*]\simeq\k[\g^{f}]$. 
Denote by ${\cal S}_e:=e+\g^{f}$ the  
{\em Slodowy slice associated  with $e$}, and let 
$T_e \colon \g \to \g, \, x \mapsto e+x$ be the translation 
map. It induces an isomorphism of affine varieties $\g^{f} \simeq 
{\cal S}_e$, and 
the comorphism $T_e^{*}$ induces an isomorphism 
between the coordinate algebras $\k[{\cal S}_e]$ and $\k[\g^{f}]$. 

Let $p$ be a homogeneous element of $\e Sg \simeq\k[\g]$. 
Then its restriction to ${\cal S}_e$ is an element of 
$\k[{\cal S}_e] \simeq \k[\g^{f}] \simeq \es S{{\goth g}^{e}}$ 
through the above isomorphisms. For $p$ in $\e Sg$, we denote by $\kappa (p)$ its 
restriction to ${\cal S}_e$ so that $\kappa(p) \in {\rm S}(\g^{e})$. Denote by $\ie{p}$ 
the initial homogeneous component of $\kappa (p)$. 
According to~\cite[Proposition~0.1]{PPY}, if $p$ is in $\ai g{}{}$, then $\ie{p}$ is in 
$\ai ge{}{}$. 

\begin{theorem}[{\cite[Theorem 1.5]{CM1}}]  \label{ti4}
Suppose that for some homogeneous generators $\poi q1{,\ldots,}{\rg}{}{}{}$ of 
$\ai g{}{}$, the polynomial functions $\poi {\ie q}1{,\ldots,}{\rg}{}{}{}$ are 
algebraically independent over $\k$. Then $e$ is a good element of ${\goth g}$. In 
particular, $\ai ge{}$ is a polynomial algebra and $\es S{{\goth g}^{e}}$ is a free 
extension of $\ai ge{}$. Moreover, $\poi {\ie q}1{,\ldots,}{\rg}{}{}{}$ is a regular 
sequence in $\es S{{\goth g}^{e}}$.
\end{theorem}

In other words, Theorem \ref{ti4} provides a sufficient condition for that  
$\ai ge{}$ is polynomial. By \cite{PPY}, one knows that  for homogeneous 
elements $q_1,\ldots,q_\ell$ of  
$\ai g{}{}$, the polynomial functions $\ie{q_1},\ldots,\ie{q_\ell}$ are algebraically 
independent if and only if 
\begin{align} \label{eq:alg_indep}
\sum_{i=1}^{\ell} \deg \ie{q_i}=\frac{\dim\g^{e}+\ell}{2}.
\end{align}
So we have a practical criterion to verify the sufficient condition of Theorem \ref{ti4}. 
However, even if the condition of Theorem \ref{ti4} holds, that is, 
if (\ref{eq:alg_indep}) holds, 
$\ai ge{}$ is not necessarily generated by the polynomial functions 
$\poi {\ie q}1{,\ldots,}{\rg}{}{}{}$. 
As a matter of fact, there are nilpotent elements $e$ satisfying this condition 
and for which ${\rm S}(\g^{e})^{\g^{e}}$ is not generated by some 
$\poi {\ie q}1{,\ldots,}{\rg}{}{}{}$, for any 
choice of homogeneous generators $\poi q1{,\ldots,}{\rg}{}{}{}$ of 
$\ai g{}{}$ (cf.~\cite[Remark 2.25]{CM1}). 

Theorem \ref{ti4} can be applied to a great number of nilpotent orbits in the simple 
classical Lie algebras, and for some nilpotent orbits in the 
exceptional Lie algebras, see \cite[Sections 5 and 6]{CM1}. 
We also provided in \cite[Example 7.8]{CM1} an example of a nilpotent element $e$ 
for which ${\rm S}(\g^{e})^{\g^{e}}$ is not polynomial, with $\g$ of type D$_7$. 

\subsection{} In this note, we prove that the converse of Theorem \ref{ti4} 
also holds. Namely, our main result is the following theorem. 

\begin{theo}\label{ti5}
The nilpotent element $e$ of ${\goth g}$ is good if and only if for some
homogeneous generating sequence $\poi q1{,\ldots,}{\rg}{}{}{}$ of $\ai g{}{}$, 
the elements $\poi {\ie q}1{,\ldots,}{\rg}{}{}{}$ are algebraically independent over $\k$.
\end{theo}

Theorem \ref{ti5} was conjectured in \cite[Conjecture 7.11]{CM1}. 
Notice that it may happen that for some $r_1,\ldots,r_\ell$ in $\ai g{}{}$, 
the elements $\poi {\ie r}1{,\ldots,}{\rg}{}{}{}$ are algebraically independent over 
$\k$, and that however $e$ is not good. This is the case for instance for the nilpotent 
elements in $\mathfrak{so}(\k^{12})$ associated with the partition $(5,3,2,2)$, 
cf.~\cite[Example 7.6]{CM1}. In fact, according to \cite[Corollary 2.3]{PPY}, for any 
nilpotent element $e$ of $\g$, there exist $\poi r1{,\ldots,}{\rg}{}{}{}$ in $\ai g{}{}$ 
such that $\poi {\ie r}1{,\ldots,}{\rg}{}{}{}$ are algebraically independent over $\k$.  
So the assumption that $\poi q1{,\ldots,}{\rg}{}{}{}$ generate $\ai g{}{}$ 
is crucial. 

\subsection{} \label{ss:plan}
We introduce in this subsection the main notations of the paper 
and we outline our strategy to prove Theorem \ref{ti5}. 

First of all, recall that ${\goth g}^{f}$ identifies with the dual of ${\goth g}^{e}$ 
through 
the Killing isomorphism so that $\es S{{\goth g}^{e}}$ is the algebra 
$\k[\g^{f}]$ of polynomial functions on ${\goth g}^{f}$, and 
that $\k[\g^{f}]$ identifies with the coordinate algebra 
of the Slodowy slice ${\cal S}_e=e+{\goth g}^{f}$.

Let $\poi x1{,\ldots,}{r}{}{}{}$ be a basis of ${\goth g}^{e}$ such that for 
$i=1,\ldots,r$, $[h,x_{i}] = n_{i}x_{i}$ with $n_{i}$ a nonnegative integer. For 
${\bf j}=(\poi j1{,\ldots,}{r}{}{}{})$ in ${\Bbb N}^{r}$, set:
$$ \vert {\bf j} \vert := \poi j1{+\cdots +}{r}{}{}{}, \quad 
\vert {\bf j} \vert _{e} := j_{1}(n_{1}+2) + \cdots + j_{r}(n_{r}+2), \quad 
x^{{\bf j}} := \poie x1{\cdots }{r}{}{}{}{j_{1}}{j_{r}}.$$
There are two gradings on $\es S{{\goth g}^{e}}$: the 
standard one and the Slodowy grading. For all ${\bf j}$ in ${\Bbb N}^{r}$, 
$x^{{\bf j}}$ has standard degree $\vert {\bf j} \vert$ and, by 
definition, it has Slodowy degree 
$\vert {\bf j} \vert_{e}$. Denoting by $t\mapsto \rho (t)$ 
the one-parameter subgroup of 
$G$ generated by $\ad h$, the Slodowy slice $e+{\goth g}^{f}$ is invariant 
under the one-parameter subgroup $t \mapsto t^{-2}\rho (t)$ of 
$G$. Hence the one-parameter subgroup $t\mapsto t^{-2}\rho (t)$ 
induces an action on $\k[{\cal S}_e]$. 
Let $j \in\{1,\ldots,r\}$, $y$ in ${\goth g}^{f}$ and $t$ in $\k^{*}$. 
Viewing the element $x_j$ of $\g^{e} \subset \es S{{\goth g}^{e}}$ 
as an element $\k[{\cal S}_e]$, we have:
$$x_{j}(t^{-2}\rho (t)(e+y)) = x_{j}(e+t^{-2}\rho (t)(y)) = 
t^{-2}\rho (t^{-1})(x_{j})(e+y) = t^{-2-n_{j}}x_{j}(e+y) ,$$  
whence for all ${\bf j}$ in ${\Bbb N}^{r}$ and for all $y$ in ${\goth g}^{f}$, 
$$ x^{{\bf j}}(t^{-2}\rho(t)(e+y)) = t^{-\vert {\bf j} \vert_{e}} x^{{\bf j}}(e+y) .$$
This means that $x^{{\bf j}}$, as a regular function on ${\cal S}_e$, 
is homogeneous of degree $\vert {\bf j} \vert_{e}$ for the Slodowy grading. 

Let $t$ be an indeterminate and let $R$ be the polynomial algebra $\k[t]$. The 
polynomial algebra 
$$\es S{{\goth g}^{e}}[t] :=\tk \k{\k[t]}\es S{{\goth g}^{e}}$$ 
identifies with the algebra of polynomial 
functions on ${\goth g}^{f}\times \k$. The grading of $\es S{{\goth g}^{e}}$ induces 
a grading of $\es S{{\goth g}^{e}}[t]$ such that $t$ has degree $0$. 
Denote by $\varepsilon $ the evaluation map at $t=0$ so that $\varepsilon $ is a graded 
morphism from $\es S{{\goth g}^{e}}[t]$ onto $\es S{{\goth g}^{e}}$. 
Let $\tau $ be the embedding of $\es S{{\goth g}^{e}}$ into $\es S{{\goth g}^{e}}[t]$ 
such that $\tau (x_{i}):=tx_{i}$ for $i=1,\ldots,r$. 

Recall that for $p$ in $\e Sg$, 
$\kappa (p)$ denotes the restriction to ${\cal S}_e$ 
of $p$ so that $\kappa(p) \in {\rm S}(\g^{e})$. 
Denote by $A$ the intersection 
of  $\es S{{\goth g}^{e}}[t]$ with the sub-$\k[t,t^{-1}]$-module of
$$\es S{{\goth g}^{e}}[t,t^{-1}]:= \tk \k{\k[t,t^{-1}]}\es S{{\goth g}^{e}}$$ 
generated by $\tau \rond \kappa(\e Sg^\g)$, 
and let $A_+$ be its augmentation ideal. 
Let ${\cal V}$ be the nullvariety of $A_{+}$ in ${\goth g}^{f}\times \k$ 
and ${\cal V}_{*}$ the union of the irreducible 
components of ${\cal V}$ which are not contained in ${\goth g}^{f}\times \{0\}$.   
Let ${\cal N}$ be the nullvariety of $\varepsilon (A)_{+}$ in ${\goth g}^{f}$, 
with $\varepsilon (A)_+$ the augmentation ideal of $\varepsilon (A)$. 
Then ${\cal V}$ is the union of ${\cal V}_{*}$ and ${\cal N}\times \{0\}$. 

The properties of the varieties ${\cal V}$ and ${\cal V}_{*}$ allow us  
to prove the following result. 

\begin{theo}\label{ti6}
Suppose that ${\cal N}$ has dimension $r-\rg$. Then for some homogeneous generating 
sequence $\poi q1{,\ldots,}{\rg}{}{}{}$ of $\ai g{}{}$, the elements 
$\poi {\ie q}1{,\ldots,}{\rg}{}{}{}$ are algebraically independent over $\k$.
\end{theo} 

The key point is to show that, under the hypothesis of Theorem \ref{ti6}, 
$\varepsilon(A)$ is the subalgebra of $\es S{{\goth g}^{e}}$ generated by the family
$\ie p,\; p \in \ai g{}{}$, and hence that ${\cal N}$ coincides with the nullvariety in 
$\g^{f}$ of $\poi {\ie q}1{,\ldots,}{\rg}{}{}{}$. So, if ${\cal N}$ has dimension 
$r -\ell$, then the elements $\poi {\ie q}1{,\ldots,}{\rg}{}{}{}$ must be algebraically 
independent over $\k$. 

The remainder of the paper is organized as follows. 
In Section~\ref{gf}, we state useful results on commutative algebra 
of independent interest. 
Some of these results are probably well-known. 
Since we have not found appropriate references, proofs are provided. 
Moreover, we formulate them as they are used in the paper. 
We study in Section \ref{sg} properties of the varieties ${\cal V}$ and ${\cal V}_{*}$.  
The proof of Theorem \ref{ti6} is achieved 
in Section \ref{sg}. Theorem \ref{ti5} is 
a consequence of Theorem \ref{ti6},  
and it is proven in Section \ref{proof}. 

\subsection*{Acknowledgments}
The second author is partially 
supported by the ANR Project GeoLie Grant number ANR-15-CE40-0012. 
We thank the referee for valuable comments and suggestions. 

\section{Some results on commutative algebra} \label{gf}
In this section $t$ is an indeterminate and the base ring $R$ is $\k$, $\k[t]$ or 
$\k[[t]]$. For $M$ a graded space over ${\Bbb N}$ and for $j$ in ${\Bbb N}$, denote by 
$M^{[j]}$ the space of degree $j$ and by $M_{+}$ the sum of $M^{[j]},j>0$. Let $A$ be a 
finitely generated graded $R$-algebra over ${\Bbb N}$ such that $A^{[0]}=R$ and such that 
$A^{[j]}$ is a free $R$-module of finite rank for any $j \in \N$. 
Moreover, $A$ is an integral domain.
Denote by $\dim A$ the Krull dimension of $A$ and 
set\footnote{Since the Lie algebra $\g$ does not appear 
in this section, there will be no possible confusion between 
$\rg$ and the rank of $\g$, denoted $\rg$, in the introduction too. 
However, the notation will be justified in the next sections.}:
$$ \rg := \left \{ \begin{array}{lcl} \dim A & \mbox{ if } & R = \k \\
\dim A -1 & \mbox{ if } & R = \k[t] \text{ or }\k[t]]. \end{array}  \right.$$
As a rule, for $B$ an integral domain, 
we denote by $K(B)$ its fraction field.

The one-dimensional multiplicative 
group of $\k$ is denoted by $\gm$. 

\subsection{} \label{gf1}
Let $B$ be a graded subalgebra of $A$. 

\begin{lemma}\label{lgf1}
\begin{enumerate}
\item[{\rm (i)}] Let $\poi {{\goth p}}1{,\ldots,}{m}{}{}{}$ be pairwise different graded
prime ideals contained in $A_{+}$. If they are the minimal prime ideals containing their
intersection, then for some homogeneous element $p$ of $A_{+}$, 
the element $p$ is not in the union 
of $\poi {{\goth p}}1{,\ldots,}{m}{}{}{}$.
\item[{\rm (ii)}] For some homogeneous sequence $\poi p1{,\ldots,}{\rg}{}{}{}$ in 
$A_{+}$, $A_{+}$ is the radical of the ideal generated by $\poi p1{,\ldots,}{\rg}{}{}{}$.
\item[{\rm (iii)}] Suppose that $A_{+}$ is the radical of $AB_{+}$. Then for some 
homogeneous sequence $\poi p1{,\ldots,}{\rg}{}{}{}$ in $B_{+}$, $A_{+}$ is the radical of
the ideal generated by $\poi p1{,\ldots,}{\rg}{}{}{}$.
\end{enumerate}
\end{lemma}

\begin{proof}
(i) Prove by induction on $j$ that for some homogeneous element $p_{j}$ of 
$A_{+}$, $p_{j}$ is not in the union of $\poi {{\goth p}}1{,\ldots,}{j}{}{}{}$.
Since ${\goth p}_{1}$ is a graded ideal strictly contained in $A_{+}$, it is true
for $j=1$. Suppose that it is true for $j-1$. If $p_{j-1}$ is not in ${\goth p}_{j}$, 
there is nothing to prove. Suppose that $p_{j-1}$ is in ${\goth p}_{j}$. According to the 
hypothesis, ${\goth p}_{j}$ is stricly contained in $A_{+}$ and it does not
contain the intersection of $\poi {{\goth p}}1{,\ldots,}{j-1}{}{}{}$. So, since
$\poi {{\goth p}}1{,\ldots,}{j}{}{}{}$ are graded ideals, for some homogeneous sequence 
$r,q$ in $A_{+}$,  
$$ r \in \bigcap _{k=1}^{j-1} {\goth p}_{k}\setminus {\goth p}_{j}, 
\quad  \text{and} \quad q \in A_{+}\setminus {\goth p}_{j} .$$
Denoting by $m$ and $n$ the respective degrees of $p_{j-1}$ and $rq$, 
$p_{j-1}^{n}+(rq)^{m}$ is homogeneous of degree $mn$ and it is not in 
$\poi {{\goth p}}1{,\ldots,}{j}{}{}{}$ since these ideals are prime.

(ii) Prove by induction on $i$ that for some homogeneous sequence 
$\poi p1{,\ldots,}{i}{}{}{}$ in $A_{+}$, the minimal prime ideals of $A$
containing $\poi p1{,\ldots,}{i}{}{}{}$ have height $i$. Let $p_{1}$ be in 
$A_{+}\setminus \{0\}$. By~\cite[Ch. 5, Theorem 13.5]{Mat}, all minimal prime ideal
containing $p_{1}$ has height $1$. Suppose that it is true for $i-1$. Let 
$\poi {{\goth p}}1{,\ldots,}{m}{}{}{}$ be the minimal prime 
ideals containing $\poi p1{,\ldots,}{i-1}{}{}{}$. Since $A_{+}$ has height $\rg > i-1$,
$A_{+}$ strictly contains $\poi {{\goth p}}1{,\ldots,}{m}{}{}{}$. By (i), there exists
a homogeneous element $p_{i}$ in $A_{+}$ not in the union of 
$\poi {{\goth p}}1{,\ldots,}{m}{}{}{}$. Then, by~\cite[Ch. 5, Theorem 13.5]{Mat}, the 
minimal prime ideals containing $\poi p1{,\ldots,}{i}{}{}{}$ have height $i$. For 
$i=\rg$, the minimal prime ideals containing $\poi p1{,\ldots,}{\rg}{}{}{}$ have height 
$\rg$. Hence they are equal to $A_{+}$ since $A_{+}$ is a prime ideal of height $\rg$ 
containing $\poi p1{,\ldots,}{\rg}{}{}{}$, whence the assertion.

(iii) The ideal $AB_{+}$ is generated by a homogeneous sequence 
$\poi a1{,\ldots,}{m}{}{}{}$ in $B_{+}$. Denote by $B'$ the subalgebra of $A$ generated 
by $\poi a1{,\ldots,}{m}{}{}{}$. Then $B'$ is a finitely generated graded subalgbera of 
$A$ such that $A_{+}$ is the radical of $AB'_{+}$. If $R=\k$, denote by $d$ its dimension
and if $t\in R$, denote by $d+1$ its dimension. By (ii),
for some homogeneous sequence $\poi p1{,\ldots,}{d}{}{}{}$ in $B'_{+}$, $B'_{+}$ is the 
radical of the ideal generated by $\poi p1{,\ldots,}{d}{}{}{}$. Then $A_{+}$ is the 
radical of the ideal of $A$ generated by $\poi p1{,\ldots,}{d}{}{}{}$. Since $A_{+}$ has 
height $\rg$, $\rg \leq d$ by~\cite[Ch. 5, Theorem 3.5]{Mat}. Since $B'$ is a subalgebra 
of $A$, its dimension is at most $\dim A$. Hence $d=\rg$.
\end{proof}

\begin{prop}\label{pgf1}
Suppose that $A_{+}$ is the radical of $AB_{+}$. Then $B$ is finitely generated and 
$A$ is a finite extension of $B$.
\end{prop}

\begin{proof}
Since $A$ is a noetherian ring, for some homogeneous sequence $\poi a1{,\ldots,}{m}{}{}{}$
in $B_{+}$, $AB_{+}$ is the ideal generated by this sequence. Denote by $C$ the 
$R$-subalgebra of $A$ generated by $\poi a1{,\ldots,}{m}{}{}{}$. Then $C$ is a graded
subalgebra of $A$. Denote by $\pi $ the morphism 
$$\xymatrix{ {\mathrm {Specm}}(A) \ar[r]^{\pi }  &  
{\mathrm {Specm}}(C)  } $$
whose comorphism is the canonical injection $C 
\hookrightarrow A$. Let $\overline{A}$ and
$\overline{C}$ be the respective 
integral closures of $A$ and $C$ in $K(A)$. Since $C$ is contained 
in $A$, $\overline{C}$ is contained in $\overline{A}$. Let $\alpha $ and $\beta $ be the 
morphisms  
$$ \xymatrix{ {\mathrm {Specm}}(\overline{A}) \ar[r]^{\alpha }  & 
{\mathrm {Specm}}(A)  } \quad  \text{and} \quad
\xymatrix{ {\mathrm {Specm}}(\overline{C}) \ar[r]^{\beta }  & 
{\mathrm {Specm}}(C)  } $$
whose comorphisms are the canonical injections 
$A\hookrightarrow \overline{A}$ and 
$C\hookrightarrow \overline{C}$ respectively. Then there is a commutative diagram
$$ \xymatrix{ {\mathrm {Specm}}(\overline{A}) \ar[rr]^{\overline{\pi }}  
\ar[d] _{\alpha } & & 
{\mathrm {Specm}}(\overline{C}) \ar[d]^{\beta } \\ 
 {\mathrm {Specm}}(A) \ar[rr]^{\pi }  & & {\mathrm {Specm}}(C)  }$$
with $\overline{\pi }$ the morphism whose comorphism is the canonical injection 
$\overline{C}\rightarrow \overline{A}$. 

The action of $\gm$ in $A$ extends to an action of $K(A)$, 
and $\overline{A}$ is 
invariant under this action. Denoting by $\overline{R}$ the integral closure of $R$ in 
$K(A)$, $\overline{R}$ is the set of fixed points under the action of $\gm$ in 
$\overline{A}$. Since $C$ is invariant under $\gm$ so is $\overline{C}$. For 
${\goth m}$ a maximal ideal of $\overline{R}$, 
the ideal ${\goth m}+\overline{C}_{+}$ is 
the maximal ideal of $\overline{C}$ containing ${\goth m}$ and invariant under 
$\gm$. Then, for ${\goth p}$ a maximal ideal of $\overline{C}$, 
${\goth p}\cap \overline{R}+\overline{C}_{+}$ is in the closure of the orbit of
${\goth p}$ under $\gm$. Moreover, 
$$ \{{\goth m}+\overline{A}_{+}\} = \overline{\pi }^{-1}\{{\goth m}+\overline{C}_{+}\} $$
for all maximal ideal ${\goth m}$ of $\overline{R}$. Hence $\overline{\pi }$ is quasi 
finite. Moreover $\overline{\pi }$ is birational. Then, by Zariski's main 
theorem~\cite{Mu}, $\overline{\pi }$ is an open immersion. 
The image of $\overline{\pi }$ contains fixed points 
for the $\gm$-action, and the closure of each $\gm$-orbit 
contains fixed points. 
As a result, $\overline{\pi }$ 
is surjective since it is $\gm$-equivariant. 
Hence $\overline{\pi }$ is an isomorphism and $\overline{A}=\overline{C}$.
As a result, $\overline{A}$ is a finite extension of $C$ since $\beta $ is a finite 
morphism. As submodules of the finite module $\overline{A}$ over the noetherian ring $C$,
$A$ and $B$ are finite $C$-modules. Hence $A$ is a finite extension of $B$. Denoting by 
$\poi {\omega }1{,\ldots,}{d}{}{}{}$ a generating family of the $C$-module $B$, $B$ is 
the subalgebra of $A$ generated by $\poi a1{,\ldots,}{m}{}{}{}$, 
$\poi {\omega }1{,\ldots,}{d}{}{}{}$.
\end{proof}

Denote by $\k[t]_{*}$ the localization of $\k[t]$ at the prime ideal $t\k[t]$ and set: 
$$ R_{*} := \left \{ \begin{array}{lcl} \k & \mbox{ if } & R = \k \\
\k[t]_{*} & \mbox{ if } & R = \k[t] \\ 
\k[[t]] & \mbox{ if } & R = \k[[t]] \end{array}
\right.  \qquad
\widehat{R} := \left \{ \begin{array}{lcl} \k & \mbox{ if } & R = \k \\
\k[[t]] & \mbox{ if } & R = \k[t] \\ 
 \k[[t]] & \mbox{ if } & R = \k[[t]]. \end{array}
\right. $$ 
For $M$ a $R$-module, set $\widehat{M} := \tk R{\widehat{R}}M$.

\begin{lemma}\label{l2gf1}
Suppose $R=\k[t]$. Let $M$ be a torsion free $R$-module and let $N$ be a submodule 
of $M$. Then for $a$ in $\widehat{N}\cap M$, $ra$ is in $N$ for some $r$ in $R$ such
that $r(0)\neq 0$.
\end{lemma}

\begin{proof}
Since $M$ is torsion free, the canonical map $M\rightarrow \widehat{M}$ is an 
embedding. Moreover, the canonical map $\widehat{N}\rightarrow \widehat{M}$ is an 
embedding since $\widehat{R}$ is flat over $R$. Let $a$ be in $\widehat{N}\cap M$  and 
let $\overline{a}$ be its image in $M/N$ by the quotient map. Denote by $J_{a}$ the 
annihilator of $\overline{a}$ in $R$, whence a commutative diagram
$$\xymatrix{
0 \ar[r] & N \ar[r]^{\dd} & M \ar[rr]^{\dd} && M/N \ar[r] & 0 \\ 
0 \ar[r] & J_{a} \ar[r]^{\dd}  & R \ar[rr]^{\dd} \ar[u]^{\delta } && 
R\overline{a} \ar[r] \ar[u]^{\delta } & 0 \\ 
& & 0 \ar[u] && 0 \ar[u] & }$$
with exact lines and columns. Since $\widehat{R}$ is a flat extension of $R$, tensoring 
this diagram by $R$ gives the following diagram with exact lines and columns:
$$\xymatrix{
0 \ar[r] & \widehat{N} \ar[r]^{\dd} & \widehat{M} 
\ar[rr]^{\dd} && \tk {R}{\widehat{R}}M/N \ar[r] & 0 \\ 
0 \ar[r] & \widehat{R}J_{a} \ar[r]^{\dd} & \widehat{R} \ar[rr]^{\dd} \ar[u]^{\delta }
&& \widehat{R}\overline{a} \ar[r] \ar[u]^{\delta } & 0 \\ 
& & 0 \ar[u] && 0 \ar[u] & }$$
For $b$ in $\widehat{R}$, $(\delta \rond \dd) b = (\dd \rond \delta) b = 0$ since $a$ is 
in $\widehat{N}$, whence $\dd b = 0$. As a result, $\widehat{R}J_{a}=\widehat{R}$. 
Hence $J_{a}$ contains an element $r$, invertible in $\widehat{R}$, that is $r(0)\neq 0$,
whence the lemma.
\end{proof}

Set 
$$A_{*} := \tk R{R_{*}}A  \qquad 
\text{and}\qquad \widehat{A} := \tk R{\widehat{R}}A. 
$$
Since $A^{[0]}=R$, the grading on $A$ extends to gradings on 
$A_{*}$ and $\widehat{A}$ such that $A_{*}^{[0]} = R_{*}$ and 
$\widehat{A}^{[0]}=\widehat{R}$. When $R=\k$ or $R=\k[[t]]$, $A_{*}=A$ and 
$\widehat{A}=A$.

For $\poi p1{,\ldots,}{\rg}{}{}{}$ a homogeneous sequence 
in $A_{+}$ set:
\begin{align*}
\underline{p} := \left \{ \begin{array}{lcl}
\poi p1{,\ldots,}{\rg}{}{}{} & \mbox{ if } & 
R = \k \\ t,\poi p1{,\ldots,}{\rg}{}{}{} & \mbox{ if } & R=\k[[t]] ,
\end{array}  \right.  
\end{align*}
and denote by $J_{\underline{p}}$ the ideal of $A$ generated by 
the sequence 
$\underline{p}$. 

\begin{lemma}\label{l3gf1}
Suppose that $A$ is Cohen-Macaulay. 
Let $\poi p1{,\ldots,}{\rg}{}{}{}$ be a homogeneous 
sequence in $A_{+}$ such that $A_{+}$ is the radical of the ideal of $A$ generated by 
$\poi p1{,\ldots,}{\rg}{}{}{}$ and let $V$ be a graded complement in $A$ to the 
$\k$-subspace $J_{\underline{p}}$.

\begin{enumerate}
\item[{\rm (i)}] The space $V$ has finite dimension.
\item[{\rm (ii)}] The space $A_{*}$ is equal to $VR_{*}[\poi p1{,\ldots,}{\rg}{}{}{}]$.
\item[{\rm (iii)}] The algebra $A$ is a flat extension of 
$R[\poi p1{,\ldots,}{\rg}{}{}{}]$.
\item[{\rm (iv)}] For all homogeneous elements $\poi a1{,\ldots,}{n}{}{}{}$ in $A$, 
linearly independent over $\k$ modulo $J_{\underline{p}}$, $\poi a1{,\ldots,}{n}{}{}{}$ 
are linearly independent over $R[\poi p1{,\ldots,}{\rg}{}{}{}]$.
\item[{\rm (v)}] The linear map
$$ \tk {\k}VR_{*}[\poi p1{,\ldots,}{\rg}{}{}{}] \longrightarrow A_{*}, \qquad 
v\tens a \longmapsto va$$ 
is an isomorphism.
\end{enumerate}
\end{lemma}

\begin{proof}
According to Lemma~\ref{lgf1}(ii), the sequence $p$ does exist.

(i) Let $J_{p}$ be the ideal of $A$ generated by $\poi p1{,\ldots,}{\rg}{}{}{}$. Since 
$A_{+}$ is the radical of $J_{p}$, $A^{[d]}=J_{p}^{[d]}$ for $d$ sufficiently big. 
When $t\in R$, for all $d$, 
then $tA^{[d]}$ has finite codimension in $A^{[d]}$ since 
$A^{[d]}$ is a finite free $R$-module. Hence $J_{\underline{p}}$ has finite codimension
in $A$ so that $V$ has finite dimension.

(ii) Suppose that $t$ is in $R$. First of all, we prove by induction on $d$ the inclusion
$$ A^{[d]} \subset (VR[\poi p1{,\ldots,}{\rg}{}{}{}])^{[d]} + tA^{[d]} .$$
Since $A^{[0]}$ is the direct sum of $V^{[0]}$ and 
$J_{\underline{p}}^{[0]}$, $V^{[0]}$ is contained in $\k + tR$, whence the inclusion for 
$d=0$. Suppose that it is true for all $j$ smaller than $d$. Since 
$\poi p1{,\ldots,}{\rg}{}{}{}$ have positive degrees, by induction hypothesis,
$$J_{\underline{p}}^{[d]} \subset (VR[\poi p1{,\ldots,}{\rg}{}{}{}])^{[d]} + tA^{[d]} ,$$ 
whence the inclusion for $d$. Then, by induction on $m$,
$$ A^{[d]} \subset (VR[\poi p1{,\ldots,}{\rg}{}{}{}])^{[d]} + t^{m}A^{[d]} .$$
As a result, since $A^{[d]}$ is a finite $R$-module, 
$$ A^{[d]} \subset (V\widehat{R}[\poi p1{,\ldots,}{\rg}{}{}{}])^{[d]} ,$$
whence $\widehat{A} = V\widehat{R}[\poi p1{,\ldots,}{\rg}{}{}{}]$. This equality remains 
true when $R=\k$ by an analogous and simpler argument.

When $R=\k[t]$, according to Lemma~\ref{l2gf1}, for $a$ in $A$, $ra$ is in 
$VR[\poi p1{,\ldots,}{\rg}{}{}{}]$ for some $r$ in $R$ such that $r(0)\neq 0$. As a 
result, $A_{*}=VR_{*}[\poi p1{,\ldots,}{\rg}{}{}{}]$.

(iii) By Proposition~\ref{pgf1}, $A$ is a finite extension of 
$R[\poi p1{,\ldots,}{\rg}{}{}{}]$. In particular, $R[\poi p1{,\ldots,}{\rg}{}{}{}]$
has dimension $\rg+\dim R$ so that $\poi p1{,\ldots,}{\rg}{}{}{}$ are algebraically 
independent over $R$. Hence $R[\poi p1{,\ldots,}{\rg}{}{}{}]$ is a regular algebra, 
whence the assertion by ~\cite[Ch. 8, Theorem 23.1]{Mat}.

(iv) Prove the assertion by induction on $n$. Since $A$ is an integral domain, 
the assertion is true for $n=1$. Suppose the assertion true for $n-1$. Let 
$(\poi b1{,\ldots,}{n}{}{}{})$ be a homogeneous sequence in 
$R[\poi p1{,\ldots,}{\rg}{}{}{}]$ such that
$$ b_{1}a_{1} + \cdots + b_{n}a_{n} = 0 .$$
Let $K$ and $I$ be the kernel and the image of the linear map
$$ R[\poi p1{,\ldots,}{\rg}{}{}{}]^{n} \longrightarrow R[\poi p1{,\ldots,}{\rg}{}{}{}], 
\qquad (\poi c1{,\ldots,}{n}{}{}{}) \longmapsto c_{1}b_{1} + \cdots + c_{n}b_{n} ,$$
whence the short exact sequence of $R[\poi p1{,\ldots,}{\rg}{}{}{}]$ modules
$$ 0 \longrightarrow K \longrightarrow R[\poi p1{,\ldots,}{\rg}{}{}{}]^{n} 
\longrightarrow I \longrightarrow 0 .$$
The grading of $R[\poi p1{,\ldots,}{\rg}{}{}{}]$ induces a grading of
$R[\poi p1{,\ldots,}{\rg}{}{}{}]^{n}$ and $K$ is a graded submodule of 
$R[\poi p1{,\ldots,}{\rg}{}{}{}]^{n}$ since $\poi b1{,\ldots,}{n}{}{}{}$ is a homogeneous 
sequence in $R[\poi p1{,\ldots,}{\rg}{}{}{}]$. Denote by $\poi y1{,\ldots,}{m}{}{}{}$ a 
generating homogeneous sequence of the $R[\poi p1{,\ldots,}{\rg}{}{}{}]$-module $K$.
By (iii), the short sequence of $A$-modules
$$ 0 \longrightarrow \tk {R[\poi p1{,\ldots,}{\rg}{}{}{}]}AK \longrightarrow 
A^{n} \longrightarrow \tk {R[\poi p1{,\ldots,}{\rg}{}{}{}]}AI \longrightarrow 0 $$
is exact. So, for some homogeneous sequence $\poi x1{,\ldots,}{m}{}{}{}$ in $A$,
$$ a_{i} = \sum_{j=1}^{m} x_{j} y_{j,i}$$
for $i=1,\ldots,n$. Since $a_{n}$ is not in $J_{\underline{p}}$, 
for some $j_{*}$, the element 
$y_{j_{*},i}$ is an invertible element of $R_{*}$, whence
$$ b_{n}y_{j_{*},n} = - \sum_{i=1}^{n-1} b_{i}y_{j_{*},i} \quad  \text{and} \quad 
\sum_{i=1}^{n-1} b_{i} (y_{j_{*},n}a_{i} - a_{n}y_{j_{*},i}) = 0 .$$
So, by induction hypothesis, 
$$\poi b1{ = \cdots = }{n-1}{}{}{} = 0$$
since the elements
$$ y_{j_{*},n}a_{1} - a_{n}y_{j_{*},1} ,\ldots,
y_{j_{*},n}a_{n-1} - a_{n}y_{j_{*},n-1}$$
are linearly independent over $\k$ modulo $J_{\underline{p}}$. Then $b_{n} = 0 $ since 
$y_{j_{*},n}$ is invertible.

(v) Let $(\poi v1{,\ldots,}{n}{}{}{})$ be a homogeneous basis of $V$. Since the space of
relations of linear dependence over $R[\poi p1{,\ldots,}{\rg}{}{}{}]$ of 
$\poi v1{,\ldots,}{n}{}{}{}$ is graded, it is equal to $\{0\}$ by (iv), whence 
the assertion by (ii).
\end{proof}

\begin{coro}\label{cgf1}
\begin{enumerate}
\item[{\rm (i)}] The algebra $A_{*}$ is Cohen-Macaulay if and only if for some 
homogeneous sequence $\poi p1{,\ldots,}{\rg}{}{}{}$ in $A_{+}$, 
the algebra $A_{*}$ is a finite free extension of $R_{*}[\poi p1{,\ldots,}{\rg}{}{}{}]$. 
\item[{\rm (ii)}] Suppose that $A_{*}$ is Cohen-Macaulay. For a homogeneous sequence
$\poi q1{,\ldots,}{\rg}{}{}{}$ in $A_{+}$, $A_{*}$ is a finite free extension
of $R_{*}[\poi q1{,\ldots,}{\rg}{}{}{}]$ if and only if $R_{*}A_{+}$ is the radical of 
the ideal of $A_{*}$ generated by $\poi q1{,\ldots,}{\rg}{}{}{}$.
\end{enumerate}
\end{coro}
 
\begin{proof}
(i) The ``only if'' part results from Lemma~\ref{l3gf1}(v). Suppose that for some 
homogeneous sequence $\poi p1{,\ldots,}{\rg}{}{}{}$ in $A_{+}$, 
the algebra $A_{*}$ is a finite free 
extension of $R_{*}[\poi p1{,\ldots,}{\rg}{}{}{}]$. In particular, 
$R_{*}[\poi p1{,\ldots,}{\rg}{}{}{}]$ is a polynomial algebra over $R_{*}$ since 
$A_{*}$ has dimension $\dim A$. Let ${\goth p}$ be a prime ideal of $A_{*}$ and let 
${\goth q}$ be its intersection with $R_{*}[\poi p1{,\ldots,}{\rg}{}{}{}]$. Denote by 
$A_{{\goth p}}$ and $R[\poi p1{,\ldots,}{\rg}{}{}{}]_{{\goth q}}$ the localizations of 
$A_{*}$ and $R_{*}[\poi p1{,\ldots,}{\rg}{}{}{}]$ at ${\goth p}$ and ${\goth q}$ 
respectively. Since $A_{*}$ is a finite extension of 
$R_{*}[\poi p1{,\ldots,}{\rg}{}{}{}]$, these local rings have 
the same dimension. Denote by $d$ this dimension. 
By flatness, any regular sequence
$\poi a1{,\ldots,}{d}{}{}{}$ in $R[\poi p1{,\ldots,}{\rg}{}{}{}]_{{\goth q}}$ is regular
in $A_{{\goth p}}$ so that $A_{{\goth p}}$ is Cohen-Macaulay. Hence $A_{*}$ is 
Cohen-Macaulay.  

(ii) The ``only if'' part results from (i) and Proposition~\ref{pgf1}. Suppose that 
$A_{*}$ is a finite free extension of $R_{*}[\poi q1{,\ldots,}{\rg}{}{}{}]$. Let 
${\goth p}$ be a minimal prime ideal of $A_{*}$ containing $\poi q1{,\ldots,}{\rg}{}{}{}$
and let ${\goth q}$ be its intersection with $R_{*}[\poi q1{,\ldots,}{\rg}{}{}{}]$. Then 
${\goth q}$ is generated by $\poi q1{,\ldots,}{\rg}{}{}{}$. In particular it has height 
$\rg$. So ${\goth p}$ has height $\rg$ since $A_{*}$ is a finite extension of 
$R_{*}[\poi q1{,\ldots,}{\rg}{}{}{}]$. As a result, ${\goth p}=R_{*}A_{+}$ since 
$R_{*}A_{+}$ is a prime ideal of height $\rg$, containing $\poi q1{,\ldots,}{\rg}{}{}{}$,
whence the assertion.
\end{proof}

Recall that $B$ is a graded subalgebra of $A$. 
Set $B_{*} := \tk R{R_{*}}B$ and for ${\goth p}$ a prime ideal of $B$, denote by 
$B_{{\goth p}}$ its localization at ${\goth p}$.

\begin{prop}\label{p2gf1}
Suppose that the following conditions are satisfied:
\begin{enumerate}
\item[{\rm (1)}] $B$ is normal,
\item[{\rm (2)}]  $A_{+}$ is the radical of $AB_{+}$,
\item[{\rm (3)}]  $A$ is Cohen-Macaulay.
\end{enumerate}
\begin{enumerate}
\item[{\rm (i)}] Let $\poi p1{,\ldots,}{\rg}{}{}{}$ be a homogeneous sequence in $B_{+}$ 
such that $B_{+}$ is the radical of the ideal of $B$ generated by this sequence. Then for
some graded subspace $V$ of $A$, having finite dimension, the linear morphisms
$$ \tk {\k}VR_{*}[\poi p1{,\ldots,}{\rg}{}{}{}] \longrightarrow A_{*}, \qquad 
v\tens a \longmapsto va,$$
$$ \tk {\k}{(V\cap B)}R_{*}[\poi p1{,\ldots,}{\rg}{}{}{}] \longrightarrow B_{*}, 
\qquad  v\tens a \longmapsto va $$
are isomorphisms.  
\item[{\rm (ii)}] If $R=\k$ or $R=\k[[t]]$, the algebra $B_{*}$ is Cohen-Macaulay.
\item[{\rm (iii)}]For ${\goth p}$ prime ideal of $B$, containing $t$, the local ring 
$B_{{\goth p}}$ is Cohen-Macaulay.
\end{enumerate}
\end{prop}

\begin{proof}
(i) By Proposition~\ref{pgf1} and by Condition (2), $B$ is finitely generated and $A$ is
a finite extension of $B$. By Condition (2) and by Lemma~\ref{lgf1}(iii), for some 
homogeneous 
sequence $\poi p1{,\ldots,}{\rg}{}{}{}$ in $B_{+}$, $A_{+}$ is the radical of the 
ideal generated by $\poi p1{,\ldots,}{\rg}{}{}{}$. 

Let $\underline{p}$ be as in Lemma \ref{l3gf1}. Denote by $m$ the degree of the extension
$K(A)$ of $K(B)$. For $a$ in $A_{*} \subset K(A)$, set: 
$$a^{\#} := \frac{1}{m} \tr a$$
with $\tr:=\tr_{K(A)/K(B)}$ 
the trace map. By Condition (1), $B_{*}$ is normal and the map 
$a\mapsto a^{\#}$ is a projection from $A_{*}$ onto $B_{*}$ whose restriction to
$A$ is a projection onto $B$. Moreover, it is a graded morphism of $B$-modules. Let $M$ 
be its kernel. Let $J_{0}$ and $J$ be the ideals of $B$ and $A$ generated by 
$\underline{p}$ respectively. Since $t,\poi p1{,\ldots,}{\rg}{}{}{}$ are in $B$, $J$ is 
the direct sum of $J_{0}$ and $MJ_{0}$. Let $V_{0}$ be a graded complement in $B$ to the
$\k$-space $J_{0}$ and let $V_{1}$ be a graded complement in $M$ to the $\k$-space
$MJ_{0}$. Setting $V:=V_{0}+V_{1}$, $V$ is a graded complement in $A$ to the 
$\k$-space $J$. By Condition (3) and Lemma~\ref{l3gf1}, $V$ has finite dimension and
the linear map
$$ \tk {\k}VR_{*}[\poi p1{,\ldots,}{\rg}{}{}{}] \longrightarrow A_{*}, \qquad 
v\tens a \longmapsto va$$
is an isomorphism. So, since $V_{0}=V^{\#}$, the linear map 
$$ \tk {\k}{V_{0}}R_{*}[\poi p1{,\ldots,}{\rg}{}{}{}] \longrightarrow B_{*}, 
\qquad v\tens a \longmapsto va$$
is an isomorphism, whence the assertion.

(ii) results from (i) and Corollary~\ref{cgf1}.

(iii) By (i) and Corollary~\ref{cgf1}, $A_{*}$ is Cohen-Macaulay. For ${\goth p}$ a
prime ideal of $B$, containing $t$, $B_{{\goth p}}$ is the localization of 
$B_{*}$ at the prime ideal $B_{*}{\goth p}$, whence the assertion by (ii).
\end{proof}

\subsection{} \label{gf2}
In this subsection $R=\k[t]$. Then $\widehat{R}=\k[[t]]$. 
For $M$ a graded module over 
$R$ such that $M^{[j]}$ is a free submodule of finite rank for all $j$, we denote by 
$P_{M,R}(T)$ its Hilbert series:
$$ P_{M,R}(T) := \sum_{j\in {\Bbb N}} \rk M^{[j]} T^{j} .$$
For $V$ a graded space over $\k$ such that $V^{[j]}$ has finite dimension, we denote 
by $P_{V,\k}(T)$ its Hilbert series:
$$ P_{V,\k}(T) := \sum_{j\in {\Bbb N}} \dim V^{[j]} T^{j} .$$
Let $S$ be a graded polynomial algebra over $\k$ such that $S^{[0]}=\k$ and 
$S^{[j]}$ has finite dimension for all $j$. Consider on $S[t]$ and $S[[t]]$ the gradings
extending that of $S$ and such that $t$ has degree $0$. 
Consider the following conditions on $A$:
\begin{itemize}
\item [{\rm (C1)}] $A$ is graded subalgebra of $S[t]$,
\item [{\rm (C2)}] for some homogeneous sequence $\poi a1{,\ldots,}{\rg}{}{}{}$ in 
$A_{+}$, $A=\k[t,t^{-1},\poi a1{,\ldots,}{\rg}{}{}{}]\cap S[t]$,
\item [{\rm (C3)}] $A$ is Cohen-Macaulay. 
\end{itemize}
If the condition (C2) holds, then $A[t^{-1}]=R[\poi a1{,\ldots,}{\rg}{}{}{}][t^{-1}]$. 
Moreover, if so, since $A$ has dimension $\rg +1$, then the elements 
$t,\poi a1{,\ldots,}{\rg}{}{}{}$ are algebraically independent over 
$\k$. Set $\widehat{A} := \tk R{\widehat{R}}A$. 

\begin{lemma}\label{lgf2} 
Assume that the conditions {\rm (C1)} and {\rm (C2)} hold. 
\begin{enumerate}
\item[{\rm (i)}] The element $t$ is a prime element of $A$.  
\item[{\rm (ii)}] The algebra $A$ is a factorial ring.
\item[{\rm (iii)}]The Hilbert series of the $R$-module $A$ is equal to
$$ P_{A,R}(T) = \prod_{i=1}^{\rg} \frac{1}{1-T^{d_{i}}},$$
with $\poi d1{,\ldots,}{\rg}{}{}{}$ the degrees of $\poi a1{,\ldots,}{\rg}{}{}{}$
respectively. 
\end{enumerate}
\end{lemma}

\begin{proof}
(i) Let $a$ and $b$ be in $A$ such that $ab$ is in $tA$. 
Since $tS[t]$ is a prime ideal
of $S[t]$, $a$ or $b$ is in $tS[t]$. Suppose $a=ta'$ for some $a'$ in $S[t]$. Then 
$a'$ is in $A[t^{-1}]$. By Condition (C2), 
$A[t^{-1}]=R[\poi a1{,\ldots,}{\rg}{}{}{}][t^{-1}]$. Hence $a'$ is in $A$ by 
Condition (C2) again. As a result, $At$ is a prime ideal of $A$. 

(ii) Since $A$ is finitely generated, it suffices to prove that all prime ideal
of height $1$ is principal by~\cite[Ch. 7, Theorem 20.1]{Mat}. Let ${\goth p}$ be 
a prime ideal of height $1$. If $t$ is in ${\goth p}$, then ${\goth p}=At$ by (i). 
Suppose that $t$ is not in ${\goth p}$ and set 
$\overline{{\goth p}}=A[t^{-1}]{\goth p}$. Then $\overline{{\goth p}}$ is a prime ideal
of height $1$ of $R[\poi a1{,\ldots,}{\rg}{}{}{}][t^{-1}]$ by Condition (C2). For $a$ in 
$\overline{{\goth p}}$, $t^{m}a$ is in ${\goth p}$ for some nonnegative integer $m$. 
Hence 
$${\goth p}=\overline{{\goth p}} \cap A $$
since ${\goth p}$ is prime. As a polynomial ring over the principal 
ring $\k[t,t^{-1}]$, the ring 
$R[\poi a1{,\ldots,}{\rg}{}{}{}][t^{-1}]$ is a factorial ring. Then 
$\overline{{\goth p}}$ is generated by an element $a$ in ${\goth p}$. Since $S$ is
a polynomial ring, $S[t]$ is a factorial ring. So, for some nonnegative integer $m$ and
for some $a'$ in $S[t]$, prime to $t$, $a=t^{m}a'$. By Condition (C2), $a'$ is in $A$. 
Then $a'$ is an element of ${\goth p}$, generating $\overline{{\goth p}}$ and not 
divisible by $t$ in $A$. Let $b$ and $c$ be in $A$ such that $bc$ is in $Aa'$. Then $b$ 
or $c$ is in $A[t^{-1}]a'$. Suppose $b$ in $A[t^{-1}]a'$. So, for some $l$ in ${\Bbb N}$, 
$t^{l}b=b'a'$ for some $b'$ in $A$. We choose $l$ minimal satisfying this condition. 
By (i), since $a'$ is not divisible by $t$ in $A$, $b'$ is divisible by $t$ in $A$ if 
$l>0$. By minimality of $l$, $l=0$ and $b$ is in $Aa'$. As a result, $Aa'$ is a prime 
ideal and ${\goth p}=Aa'$ since ${\goth p}$ has height $1$.

(iii) By Condition (C2), 
$$ A[t^{-1}] = \tk {\k}{\k[t,t^{-1}]}\k[\poi a1{,\ldots,}{\rg}{}{}{}] \quad  
\text{whence} \quad \rk A^{[d]} = \dim \k[\poi a1{,\ldots,}{\rg}{}{}{}]^{[d]}$$
for all nonnegative integer $d$. Since $\poi a1{,\ldots,}{\rg}{}{}{}$ are algebraically 
independent over $\k$,
$$ P_{\k[\poi a1{,\ldots,}{\rg}{}{}{}],\k}(T) = \prod_{i=1}^{\rg} \frac{1}{1-T^{d_{i}}} 
,$$
whence the assertion.
\end{proof}

Let $\poi p1{,\ldots,}{\rg}{}{}{}$ be a homogeneous sequence in $A$ such that $A_{+}$ is 
the radical of the ideal of $A$ generated by this sequence. By Lemma~\ref{lgf1}(ii), 
such a sequence does exist. Denote by $C$ the integral closure of 
$\k[\poi p1{,\ldots,}{\rg}{}{}{}]$ in $\k(t,\poi a1{,\ldots,}{\rg}{}{}{})$. 

\begin{lemma}\label{l2gf2} 
Assume that the conditions {\rm (C1)}, {\rm (C2)} and {\rm (C3)} hold. 
\begin{enumerate}
\item[{\rm (i)}] The algebra $C$ is a graded subalgebra of $A$ and $t$ is not algebraic 
over $C$.
\item[{\rm (ii)}] The algebra $C$ is Cohen-Macaulay. Moreover, $C$ is a finite free 
extension of $\k[\poi p1{,\ldots,}{\rg}{}{}{}]$.
\item[{\rm (iii)}] The algebra $C+tA$ is normal.
\end{enumerate}
\end{lemma}

\begin{proof}
(i) By Lemma~\ref{lgf2}(ii), $A$ is a normal ring such that 
$K(A)=\k(t,\poi a1{,\ldots,}{\rg}{}{}{})$ by Condition (C2). Then $C$ is contained in 
$A$ since $\k[\poi p1{,\ldots,}{\rg}{}{}{}]$ is contained in $A$. Moreover, 
$C$ is a graded algebra since so is $\k[\poi p1{,\ldots,}{\rg}{}{}{}]$. By Proposition
\ref{pgf1}, $A$ is a finite extension of $R[\poi p1{,\ldots,}{\rg}{}{}{}]$. So, since
$A$ has dimension $\rg+1$, 
the elements 
$t,\poi p1{,\ldots,}{\rg}{}{}{}$ are algebraically independent
over $\k$. As a result, $t$ is not algebraic over $C$.

(ii) By (i), $C[[t]]=\tk {\k}{C}\k[[t]]$ so that $C[[t]]$ is a flat extension of 
$\k[[t]]$. Moreover, $C$ is the quotient of $C[[t]]$ by $tC[[t]]$. As $C$ and $\k[[t]]$ 
are normal rings, $C[[t]]$ is a normal ring by 
\cite[Ch. 8, Corollary of Theorem 23.9]{Mat}. By definition, 
$A_{+}$ is the radical of the ideal of $A$ generated by $\poi p1{,\ldots,}{\rg}{}{}{}$.
As $\k[[t]]$ is a flat extension of $\k[t]$, from the short exact sequence
$$ \xymatrix{ 0 \ar[r] & A_{+} \ar[r] & A \ar[r] & \k[t] \ar[r] & 0}$$
we deduce the short exact sequence
$$ \xymatrix{ 0 \ar[r] & \widehat{A}_{+} \ar[r] & \widehat{A} \ar[r] & 
\k[[t]] \ar[r] & 0}.$$
Hence $\widehat{A}_{+}$ is a prime ideal. As $A_{+}$ is the radical of the ideal generated
by the sequence $\poi p1{,\ldots,}{\rg}{}{}{}$, $\widehat{A}_{+}$ is contained in
the radical of $AC[[t]]_{+}$. Then, by (i), $\widehat{A}_{+}$ is the radical of 
$AC[[t]]_{+}$. Since $\widehat{R}$ is a flat extension of $R$, the algebra 
$\widehat{A}$ is Cohen-Macaulay by Condition~(C3). Then, by 
Proposition~\ref{p2gf1}(ii), 
$C[[t]]$ is Cohen-Macaulay. Let $V$ be a graded complement in $C$ to the ideal
of $C$ generated by $\poi p1{,\ldots,}{\rg}{}{}{}$. Since $t$ is not algebraic over 
$C$, the space $V$ is a complement in $C[t]$ to the ideal of $C[t]$ generated by 
$t,\poi p1{,\ldots,}{\rg}{}{}{}$. Then, by Lemma~\ref{l3gf1}, $V$ has finite dimension
and the linear morphism 
$$ \tk {\k}VR_{*}[\poi p1{,\ldots,}{\rg}{}{}{}] \longrightarrow R_{*}C, \qquad 
v\tens a \longmapsto va$$
is an isomorphism. As a result, the linear morphism 
$$ \tk {\k}V\k[\poi p1{,\ldots,}{\rg}{}{}{}] \longrightarrow C, \qquad 
v\tens a \longmapsto va$$
is an isomorphism, whence the assertion by Corollary~\ref{cgf1}(ii).

(iii) Set $\tilde{A} := C+tA$. At first, $\tilde{A}$ is a graded subalgebra of $A$ since 
$C$ is a graded algebra and $tA$ is a graded ideal of $A$. According to 
Proposition~\ref{p2gf1}(i), for some graded subspace $V$ of $A$, 
having finite dimension, the linear morphisms
$$ \tk {\k}VR_{*}[\poi p1{,\ldots,}{\rg}{}{}{}] \longrightarrow A_{*}, 
\qquad 
v\tens a \longmapsto va,$$
$$ \tk {\k}{(V\cap C[t])}R_{*}[\poi p1{,\ldots,}{\rg}{}{}{}] \longrightarrow R_{*}C, 
\qquad v\tens a \longmapsto va $$
are isomorphisms. Let $\poi v1{,\ldots,}{n}{}{}{}$ be a basis of $V$ such that
$\poi v1{,\ldots,}{m}{}{}{}$ is a basis of $V\cap C[t]$. For $a$ in $A_{*}$, 
the element $a$ 
has unique expansion
$$ a = v_{1}a_{1} + \cdots + v_{n}a_{n}$$
with $\poi a1{,\ldots,}{n}{}{}{}$ in $R_{*}[\poi p1{,\ldots,}{\rg}{}{}{}]$. If $a$ 
is in $tA_{*}$, $\poi a1{,\ldots,}{n}{}{}{}$ are in 
$tR_{*}[\poi p1{,\ldots,}{\rg}{}{}{}]$ and if $a$ is in $R_{*}C$, 
$\poi a1{,\ldots,}{m}{}{}{}$ are in $\k[\poi p1{,\ldots,}{\rg}{}{}{}]$
and $\poi a{m+1}{,\ldots,}{n}{}{}{}$ are equal to $0$, whence 
$R_{*}C\cap tA_{*}=tR_{*}C$ and 
$C\cap tA = \{0\}$. In particular, $C$ is the quotient of $\tilde{A}$ by $t\tilde{A}$.

For ${\goth p}$ a prime ideal of $\tilde{A}$, denote by $\tilde{A}_{{\goth p}}$ the 
localization of $\tilde{A}$ at ${\goth p}$. If $t$ is not in ${\goth p}$, then 
$A[t^{-1}]$ is contained in $\tilde{A}_{{\goth p}}$ so that $\tilde{A}_{{\goth p}}$
is a localization of the regular algebra $R[\poi a1{,\ldots,}{\rg}{}{}{}][t^{-1}]$
by Condition (C2). Hence $\tilde{A}_{{\goth p}}$ is a regular local algebra. Suppose
that $t$ is in ${\goth p}$. Denote by $\overline{{\goth p}}$ the image of ${\goth p}$
in $C$ by the quotient map. Then $\tilde{A}_{{\goth p}}/t\tilde{A}_{{\goth p}}$ is 
the localization $C_{\overline{{\goth p}}}$ of $C$ at the prime ideal 
$\overline{{\goth p}}$. Since $C$ is Cohen-Macaulay, so are $C_{\overline{{\goth p}}}$
and $\tilde{A}_{{\goth p}}$. As a result, $\tilde{A}$ is Cohen-Macaulay.

Let ${\goth p}$ be a prime ideal of height $1$ of $\tilde{A}$. If $t$ is not in 
${\goth p}$,  $\tilde{A}_{{\goth p}}$ is a regular local algebra as it is already
mentioned. Suppose that $t$ is in ${\goth p}$. By Lemma~\ref{lgf2}(i), 
$t\tilde{A}={\goth p}$ so that all element of $C\setminus \{0\}$ is invertible in
$\tilde{A}_{{\goth p}}$, whence 
$$\tilde{A}_{{\goth p}}=K(C)+t\tilde{A}_{{\goth p}} \quad  \text{and} \quad 
t\tilde{A}_{{\goth p}} = tK(C) + t^{2}\tilde{A}_{{\goth p}} .$$
Hence $\tilde{A}_{{\goth p}}$ is a regular local ring of dimension $1$. As a result,
$\tilde{A}$ is regular in codimension $1$. Then, by Serre's normality
criterion~\cite[\S 1, \no 10, Th\'eor\`eme 4]{Bou}, $\tilde{A}$ is normal since 
$\tilde{A}$ is Cohen-Macaulay.
\end{proof}

\begin{coro}\label{cgf2}
Assume that the conditions {\rm (C1)}, {\rm (C2)} and {\rm (C3)} hold. 
\begin{enumerate}
\item[{\rm (i)}] The algebra $\widehat{A}$ is equal to $C[[t]]$.
\item[{\rm (ii)}] For $a$ in $A$, the element $ra$ is in $C[t]$ for some $r$ in $\k[t]$ 
such that $r(0) \neq 0$.
\end{enumerate}
\end{coro}

\begin{proof}
(i) Since $tA$ is contained in $A$, we have $K(A)=K(\tilde{A})$. Since $C_{+}$ is 
contained in $\tilde{A}_{+}$, $A_{+}$ is the radical of $A\tilde{A}_{+}$. Then, by 
Proposition~\ref{pgf1}, $A$ is a finite extension of $\tilde{A}$. So, by 
Lemma~\ref{l2gf2}(iii), $A=\tilde{A}$ and by induction on $m$, 
$$ A \subset C[t] + t^{m}A$$
for all positive integer $m$. Since $A$ and $C[t]$ are graded and since the $R$-module 
$A^{[d]}$ is finitely generated for all $d$, $\widehat{A}=C[[t]]$.

(ii) The assertion results from (i) and Lemma~\ref{l2gf1}.
\end{proof}

\begin{prop}\label{pgf2} 
Assume that the conditions {\rm (C1)}, {\rm (C2)} and {\rm (C3)} hold. Then the algebra 
$A_{*}$ is polynomial over $R_{*}$. Moreover, for some homogeneous sequence 
$\poi q1{,\ldots,}{\rg}{}{}{}$ in $A_{+}$ such that $\poi q1{,\ldots,}{\rg}{}{}{}$ have 
degree $\poi d1{,\ldots,}{\rg}{}{}{}$ respectively, 
$A_{*}=R_{*}[\poi q1{,\ldots,}{\rg}{}{}{}]$.
\end{prop}

\begin{proof}
According to Corollary~\ref{cgf2} and Lemma~\ref{l2gf2}(i), it suffices to prove
that $C$ is a polynomial algebra over $\k$ generated by a homogeneous sequence 
$\poi q1{,\ldots,}{\rg}{}{}{}$ such that $\poi q1{,\ldots,}{\rg}{}{}{}$ have degree
$\poi d1{,\ldots,}{\rg}{}{}{}$ respectively. According to Corollary~\ref{cgf2}(i) 
Lemma~\ref{l2gf2}(i) and Lemma~\ref{lgf2}(iii),
$$ P_{C,\k}(T) = \prod_{i=1}^{\rg} \frac{1}{1-T^{d_{i}}} .$$
By Corollary~\ref{cgf2}(ii), for $i=1,\ldots,\rg$, for some $r_{i}$ in $R$ such 
that $r_{i}(0)\neq 0$, $r_{i}a_{i}$ has an expansion
$$ r_{i}a_{i} = \sum_{m\in {\Bbb N}} c_{i,m} t^{m}$$
with $c_{i,m},m\in {\Bbb N}$ in $C^{[d_{i}]}$, with finite support. For $z$ in $\k$ and 
$i=1,\ldots,\rg$, set:
$$ b_{i}(z) = \sum_{m\in {\Bbb N}} c_{i,m} z^{m} $$
so that $b_{i}(z)$ is in $C^{[d_{i}]}$ for all $z$. As already mentioned, 
$t,\poi a1{,\ldots,}{\rg}{}{}{}$ are algebraically independent over $\k$ by
Condition (C2) since $A$ has dimension $\rg +1$. Then, so are 
$t,r_{1}a_{1},\ldots,r_{\rg}a_{\rg}$ and for some $z$ in $\k$, 
$\poi z{}{,\ldots,}{}{b}{1}{\rg}$ are algebraically independent over $\k$. Denoting
by $C'$ the subalgebra of $C$ generated by this sequence,
$$ P_{C',\k}(T) = \prod_{i=1}^{\rg} \frac{1}{1-T^{d_{i}}},$$
whence $C=C'$ so that $C$ is a polynomial algebra.
\end{proof}

\section{Proof of Theorem \ref{ti6}} \label{sg}
In this section, unless otherwise specified, the grading on $\es S{{\goth g}^{e}}$ is 
the Slodowy grading. 
 
For $m$ a nonnegative integer, $\es S{{\goth g}^{e}}^{[m]}$ denotes the space of degree 
$m$ of $\es S{{\goth g}^{e}}$. We retain the notations of the introduction, 
in particular of Subsection \ref{ss:plan}. 

\subsection{} \label{sg1}
Let $R$ be the ring $\k[t]$. As in Section~\ref{gf}, for $M$ a graded subspace of 
$\es S{{\goth g}^{e}}[t]=\tk \k{R}\es S{{\goth g}^{e}}$, its subspace of degree $m$ is 
denoted by $M^{[m]}$. In particular, $\es S{{\goth g}^{e}}[t]^{[m]}$ is equal to 
$\es S{{\goth g}^{e}}^{[m]}[t]$ and it is a free $R$-module of finite rank. As a result, 
for all graded $R$-submodule $M$ of $\es S{{\goth g}^{e}}[t]$, its Hilbert series is well
defined. 

For $m$ a nonnegative integer, denote by $F_{m}$ the space of elements of 
$\kappa (\ai g{}{})$ whose component of minimal standard degree is at least $m$.
Then $\poi F0{,}{1}{}{}{},\ldots$ is a decreasing filtration of the algebra 
$\kappa (\ai g{}{})$. Let $\poi d1{,\ldots,}{\rg}{}{}{}$ be the standard degrees
of a homogeneous generating sequence of $\ai g{}{}$. We assume that the sequence 
$\poi d1{,\ldots,}{\rg}{}{}{}$ is increasing. 

Recall that $A$ is the intersection of $\es S{{\goth g}^{e}}[t]$ with the 
sub-$\k[t,t^{-1}]$-module of $\es S{{\goth g}^{e}}[t,t^{-1}]$ 
generated by $\tau \rond \kappa(\e Sg^\g)$, and that $A_+$ is the augmentation ideal of 
$A$.

\begin{lemma}\label{lsg1}
\begin{enumerate}
\item[{\rm (i)}] For $p$ a homogeneous element of standard degree $d$ in $\ai g{}{}$, the
element $\kappa (p)$ and $\ie p$ have degree $2d$.
\item[{\rm (ii)}] For some homogeneous sequence $\poi a1{,\ldots,}{\rg}{}{}{}$ in 
$A_{+}$, the elements $t,\poi a1{,\ldots,}{\rg}{}{}{}$ are algebraically independent over
$\k$, and $A$ is the intersection of ${\rm S}({\goth g}^{e})[t]$ with  
$\k[t,t^{-1},\poi a1{,\ldots,}{\rg}{}{}{}]$.
\item[{\rm (iii)}]The Hilbert series of the $R$-algebra $A$ is equal to
$$ P_{A,R}(T) = \prod_{i=1}^{\rg} \frac{1}{1-T^{2d_{i}}} .$$
\item[{\rm (iv)}] The Hilbert series of the $\k$-algebra $\varepsilon (A)$ is equal to
$$ P_{\varepsilon (A),\k}(T) = \prod_{i=1}^{\rg} \frac{1}{1-T^{2d_{i}}} .$$
\item[{\rm (v)}] The subalgebra $\varepsilon (A)$ is the graded algebra associated with 
the filtration $\poi F0{,}{1}{}{}{},\ldots$.
\end{enumerate}
\end{lemma}

\begin{proof}
(i) Let $\rho$ be as in Subsection \ref{ss:plan}. 
For $y$ in ${\goth g}^{f}$ and $s$ in $\k^{*}$,
$$ p(s^{-2}\rho (s)(e+y)) = s^{-2d}p(\rho (s)(e+y)) = s^{-2d}p(e+y)$$
since $p$ is invariant under the one-parameter subgroup $\rho$. Hence 
$\kappa (p)$ is homogeneous of degree $2d$. 
Since the monomials $x^{{\bf j}}$ are homogeneous, $\ie p$ has degree $2d$. 

(ii) Let $\poi q1{,\ldots,}{\rg}{}{}{}$ be a homogeneous generating sequence of 
$\ai g{}{}$. By a well known fact 
(cf.~e.g.~\cite[Lemma 4.4(i)]{CM1}), the morphism 
$$ G \times (e+{\goth g}^{f}) \longrightarrow {\goth g}, \qquad (g,x) \longmapsto g(x)$$
is dominant. Then $\kappa (\ai g{}{})$ is a polynomial algebra generated by 
$\poi q1{,\ldots,}{\rg}{\kappa }{}{}$. So, setting $a_{i}:=\tau \rond \kappa (q_{i})$
for $i=1,\ldots,\rg$, the sequence $\poi a1{,\ldots,}{\rg}{}{}{}$ is a homogeneous 
sequence in $A_{+}$ such that 
$$\tau \rond \kappa (\ai g{}{})[t,t^{-1}] = \k[t,t^{-1},\poi a1{,\ldots,}{\rg}{}{}{}] .$$
Let $\overline{\tau }$ be the automorphism of $\es S{{\goth g}^{e}}[t,t^{-1}]$ extending 
$\tau $ and such that $\overline{\tau }(t)=t$. Then 
$$ \tau \rond \kappa (\ai g{}{})[t,t^{-1}] = 
\overline{\tau }(\kappa (\ai g{}{})[t,t^{-1}]) .$$  
Since $\kappa (\ai g{}{})[t,t^{-1}]$ has dimension $\rg +1$, 
$\tau \rond \kappa (\ai g{}{})[t,t^{-1}]$ has dimension $\rg +1$ too, 
and $t,\poi a1{,\ldots,}{\rg}{}{}{}$ are 
algebraically independent over $\k$. By definition, 
$A=\es S{{\goth g}^{e}}[t]\cap \tau \rond \kappa (\ai g{}{})[t,t^{-1}]$. Hence
$$ A[t^{-1}] = \k[t,t^{-1},\poi a1{,\ldots,}{\rg}{}{}{}] \quad  \text{and} \quad
A = \es S{{\goth g}^{e}}[t] \cap \k[t,t^{-1},\poi a1{,\ldots,}{\rg}{}{}{}] .$$

(iii) Since $t$ has degree $0$, the grading of $\es S{{\goth g}^{e}}[t]$ extends to 
a grading of $\es S{{\goth g}^{e}}[t,t^{-1}]$ such that for all $m$, its space of 
degree $m$ is equal to $\es S{{\goth g}^{e}}^{[m]}[t,t^{-1}]$. Then for all 
$\k[t,t^{-1}]$-submodule $M$ of $\es S{{\goth g}^{e}}[t,t^{-1}]$, $M$ has a 
Hilbert series:
$$ P_{M,\k[t,t^{-1}]}(T) := \sum_{m\in {\Bbb N}} \rk M^{[m]} T^{m}$$
with $M^{[m]}$ the subspace of degree $m$ of $M$. From the equality
$A[t^{-1}] = \k[t,t^{-1},\poi a1{,\ldots,}{\rg}{}{}{}]$, we deduce
$$ P_{A[t^{-1}],\k[t,t^{-1}]}(T) = \prod_{i=1}^{\rg} \frac{1}{1-T^{2d_{i}}}$$
since for $i=1,\ldots,\rg$, the element $a_{i}$ has degree $2d_{i}$ by (i). For all $m$, 
the rank of the $R$-module $A^{[m]}$ is equal to the rank of the $\k[t,t^{-1}]$-module 
$A[t^{-1}]^{[m]}$, whence
$$ P_{A,R}(T) = \prod_{i=1}^{\rg} \frac{1}{1-T^{2d_{i}}} .$$

(iv) Let $m$ be a nonnegative integer. The $R$-module $A^{[m]}$ is free of 
finite rank and for $(\poi v1{,\ldots,}{n}{}{}{})$ a basis of this module, 
$(\poi {tv}1{,\ldots,}{n}{}{}{})$ is a basis of the $R$-module $tA^{[m]}$. Since 
$\varepsilon (A)^{[m]}$ is the quotient of $A^{[m]}$ by $tA^{[m]}$, 
$$ \dim \varepsilon (A)^{[m]} = n = \rk A^{[m]},$$
whence the assertion by (iii).

(v) Let ${\rm gr}_FA$ be the graded algebra associated with the filtration 
$\poi F0{,}{1}{}{}{},\ldots$ of $\kappa (\ai g{}{})$. 
Denote by $a\mapsto a(1)$ the evaluation map at $t=1$ from $\es S{{\goth g}^{e}}[t]$ to 
$\es S{{\goth g}^{e}}$. For $a$ in $A$ such that $\varepsilon (a)\neq 0$, $a(1)$ is in 
$\kappa (\ai g{}{})$ and $\varepsilon (a)$ is the component of minimal degree of $a(1)$ 
with respect to the standard grading, whence $\varepsilon (A)\subset {\rm gr}_F A$. 
Conversely, let $\overline{a}$ be a homogeneous element of degree $m$ of ${\rm gr}_F A$ 
and let $a$ be a representative of $\overline{a}$ in $F_{m}$. Then $\tau (a)=t^{m}b$ with
$b$ in $A$ such that $\varepsilon (b)=\overline{a}$, whence 
${\rm gr}_F A \subset \varepsilon (A)$ and the assertion.  
\end{proof}

Let $R_{*}$ be the localization of $R$ at the prime ideal $tR$ and set 
$$  \widehat{R} := \k[[t]], \qquad A_{*} := \tk R{R_{*}}A, \qquad 
\widehat{A} := \tk R{\widehat{R}}A .$$
The grading of $A$ extends to gradings on $A_{*}$ and $\widehat{A}$
such that $A_{*}^{[0]}=R_{*}$ and $\widehat{A}^{[0]}=\widehat{R}$. 

\begin{prop}\label{psg1}
\begin{enumerate}
\item[{\rm (i)}] The algebra $\varepsilon (A)$ is polynomial if and only if 
for some standard homogeneous generating sequence $\poi q1{,\ldots,}{\rg}{}{}{}$ of 
$\ai g{}{}$, the elements $\poi {\ie q}1{,\ldots,}{\rg}{}{}{}$ are algebraically 
independent over $\k$. Moreover, in this case, $A$ is a polynomial algebra. 
\item[{\rm (ii)}] If $A_{*}$ is a polynomial algebra over $R_{*}$, then for some 
homogeneous sequence $\poi p1{,\ldots,}{\rg}{}{}{}$ in $A_{+}$, we 
have $A_{*} = R_{*}[\poi p1{,\ldots,}{\rg}{}{}{}]$, the elements 
$t,\poi p1{,\ldots,}{\rg}{}{}{}$ are algebraically independent over $\k$ and 
$\poi p1{,\ldots,}{\rg}{}{}{}$ have degree $\poi {2d}1{,\ldots,}{\rg}{}{}{}$ 
respectively.
\end{enumerate}
\end{prop}

\begin{proof}
(i) Let $\poi q1{,\ldots,}{\rg}{}{}{}$ be a homogeneous generating sequence of 
$\ai g{}{}$ such that $\poi {\ie q}1{,\ldots,}{\rg}{}{}{}$ are algebraically 
independent over $\k$. We can assume that for $i=1,\ldots,\rg$, $q_i$ has standard degree
$d_i$. For $i=1,\ldots,\rg$, $\ie{q_i}$ has degree $2 d_i$ by Lemma~\ref{lsg1}(i), and 
we set 
$$Q_{i} := t^{- 2 d_i}\tau \rond \kappa (q_{i}).$$ 
Then $Q_i$, for $i=1,\ldots,\rg$, is in $A$ by definition of $A$. 
For ${\bf i}=(\poi i1{,\ldots,}{\rg}{}{}{})$ in ${\Bbb N}^{\rg}$, set:
$$ q^{{\bf i}} := \poie q1{\cdots}{\rg}{}{}{}{i_{1}}{i_{\rg}}, \qquad
Q^{{\bf i}} := \poie Q1{\cdots}{\rg}{}{}{}{i_{1}}{i_{\rg}}, \qquad 
\ie q^{{\bf i}} := \poie {\ie q}1{\cdots}{\rg}{}{}{}{i_{1}}{i_{\rg}},$$
$$\vert {\bf i} \vert_{{\mathrm {min}}} := 
2 i_{1} d _{1} + \cdots + 2 i_{\rg} d _{\rg} .$$
Then, for all ${\bf i}$ in ${\Bbb N}^{\rg}$,
$$ \tau \rond \kappa (q^{{\bf i}}) = t^{\vert {\bf i} \vert_{{\mathrm {min}}}} 
Q^{{\bf i}}.$$
Moreover, 
$$\tau \rond \kappa (\ai g{}{})[t,t^{-1}] = \k[t,t^{-1},\poi Q1{,\ldots,}{\rg}{}{}{}]. $$
Let $a$ be in $A$. For some $l$ in ${\Bbb N}$ and for some sequence 
$c_{{\bf i},m}, \; ({\bf i},m)\in {\Bbb N}^{\rg}\times {\Bbb N}$ in $\k$, 
of finite support,
$$ t^{l}a = \sum_{({\bf i},m) \in {\Bbb N}^{\rg}\times {\Bbb N}}
c_{{\bf i},m} t^{m}Q^{{\bf i}} \quad  \text{whence} \quad 
\sum_{{\bf i}\in {\Bbb N}^{\rg}} c_{{\bf i},m} \ie q^{{\bf i}} = 0 $$
for $m<l$. Hence $a$ is in $R[\poi Q1{,\ldots,}{\rg}{}{}{}]$ since the elements 
$\ie q^{{\bf i}}$, ${\bf i}\in {\Bbb N}^{\rg}$ are linearly independent over $\k$. As a 
result,
$$ A = R[\poi Q1{,\ldots,}{\rg}{}{}{}] \quad  \text{and} \quad 
\varepsilon (A) = \k[\poi {\ie q}1{,\ldots,}{\rg}{}{}{}] $$
so that $A$ and $\varepsilon (A)$ are polynomial algebras over $\k$ since 
$\poi {\ie q}1{,\ldots,}{\rg}{}{}{}$ are algebraically independent over $\k$. 

Conversely, suppose that $\varepsilon (A)$ is a polynomial algebra. By 
Lemma~\ref{lsg1}, (i) and (iv), the algebra 
$\varepsilon (A)$ is graded for both 
Slodowy grading and standard grading. Let 
$d$ be the dimension of $\varepsilon (A)$. As $\varepsilon (A)$ is a polynomial algebra,
it is regular so that the $\k$-space $\varepsilon (A)_{+}/\varepsilon (A)_{+}^{2}$ has 
dimension $d$. Moreover, the two gradings on $\varepsilon (A)$ induce gradings on 
$\varepsilon (A)_{+}/\varepsilon (A)_{+}^{2}$. Hence 
$\varepsilon (A)_{+}/\varepsilon (A)_{+}^{2}$ has a bihomogeneous basis. Then some
bihomogeneous sequence $\poi u1{,\ldots,}{d}{}{}{}$ in $\varepsilon (A)_{+}$ represents
a basis of $\varepsilon (A)_{+}/\varepsilon (A)_{+}^{2}$. As a result, the $\k$-algebra
$\varepsilon (A)$ is generated by the bihomogeneous sequence $\poi u1{,\ldots,}{d}{}{}{}$.
For $i=1,\ldots,d$, denote by $\delta _{i}$ the Slodowy degree of $u_{i}$. As 
$\varepsilon $ is homogeneous with respect to the Slodowy grading, 
$u_{i}=\varepsilon (r_{i})$ for some homogeneous element $r_{i}$ of degree $\delta _{i}$
of $A$. Let $m_{i}$ be the smallest nonnegative integer such that $t^{m_{i}}r_{i}$ is 
in $\tau \rond \kappa (\ai g{}{})$. According to Lemma~\ref{lsg1}(i), $\delta _{i}$ is
even and for some standard homogeneous element $p_{i}$ of standard degree $\delta _{i}/2$
of $\ai g{}{}$, $t^{m_{i}}r_{i}=\tau \rond \kappa (p_{i})$. Then $u_{i}=\ie {p_{i}}$ 
since $p_{i}$ is standard homogeneous.

Let ${\goth P}$ be the subalgebra of $\e Sg$ generated by $\poi p1{,\ldots,}{d}{}{}{}$.
Suppose that ${\goth P}$ is strictly contained in $\ai g{}{}$. A contradiction is 
expected. For some positive integer $m$, the space $\ai g{}{}_{m}$ of standard degree 
$m$ of $\ai g{}{}$ is not contained in ${\goth P}$. Let $q$ be in 
$(\ai g{}{})_{m}\setminus {\goth P}$ such that $\ie q$ has maximal standard degree. By 
Lemma~\ref{lsg1}(i), $\ie q$ is a polynomial in $\poi u1{,\ldots,}{d}{}{}{}$, of degree 
$2m$. So, for some polynomial $q'$ of degree $m$ in ${\goth P}$, $\ie (q-q')$ has 
standard degree bigger than the standard degree of $\ie q$. So, by maximality of the 
standard degree of $\ie q$, the elements $q-q'$ and $q$ are in ${\goth P}$, whence the 
contradiction. As a result, ${\goth P}=\ai g{}{}$ and $d=\rg$.

(ii) Suppose that $A_{*}$ is a polynomial algebra. Denoting by $J$ the ideal of $A_{*}$
generated by $t$ and $A_{+}$, the $\k$-space $J/J^{2}$ is a graded space of dimension 
$\rg +1$ since $A_{*}$ is a regular algebra of dimension $\rg+1$. Then for some 
homogeneous sequence $\poi p1{,\ldots,}{\rg}{}{}{}$ in $A_{+}$, 
$(t,\poi p1{,\ldots,}{\rg}{}{}{})$ is a basis of $J$ modulo $J^{2}$. Since 
$\poi p1{,\ldots,}{\rg}{}{}{}$ have positive degree, we prove by induction on $d$ that 
$$ A_{*}^{[d]} \subset R_{*}[\poi p1{,\ldots,}{\rg}{}{}{}]^{[d]} + tA_{*}^{[d]} .$$ 
Then by induction on $m$, we get 
$$ A_{*}^{[d]} \subset R_{*}[\poi p1{,\ldots,}{\rg}{}{}{}] + t^{m}A_{*}^{[d]} .$$
So, since the $R_{*}$-module $A_{*}^{[d]}$ is finitely generated,
$$ A_{*}^{[d]} \subset \widehat{R}[\poi p1{,\ldots,}{\rg}{}{}{}]^{[d]}.$$
Apply Lemma~\ref{l2gf1} to $N=A$ and $M=\es S{{\goth g}^{e}}[t]$. 
Since $\widehat{N}=\widehat{R}[\poi p1{,\ldots,}{\rg}{}{}{}]$, for $a \in N$,  
there exists $r \in R$ such that $r(0)\neq 0$ and 
$ra \in R[\poi p1{,\ldots,}{\rg}{}{}{}]$ by Lemma~\ref{l2gf1}. So 
$A_{*}$ is contained  in $R_{*}[\poi p1{,\ldots,}{\rg}{}{}{}]$, 
whence $A_{*}=R_{*}[\poi p1{,\ldots,}{\rg}{}{}{}]$.

Denote by $\poi {\delta }1{,\ldots,}{\rg}{}{}{}$ the respective degrees of 
$\poi p1{,\ldots,}{\rg}{}{}{}$. We can suppose that $\poi p1{,\ldots,}{\rg}{}{}{}$ is 
ordered so that $\poi {\delta }1{\leq \cdots \leq }{\rg}{}{}{}$. Prove by 
induction on $i$ that $\delta _{j}=2d_{j}$ for $j=1,\ldots,i$. By Lemma~\ref{lsg1}(iii),
$2d_{1}$ is the smallest positive degree of the elements of $A$. Moreover, 
$\delta _{1}$ is the smallest positive degree of the elements of 
$R[\poi p1{,\ldots,}{\rg}{}{}{}]$, whence $\delta _{1}=2d_{1}$. Suppose 
$\delta _{j}=2d_{j}$ for $j=1,\ldots,i-1$. Set 
$A_{i} := R[\poi pi{,\ldots,}{\rg}{}{}{}]$. Then, by induction hypothesis and 
Lemma~\ref{lsg1}(iii), 
$$ P_{A_{i},R}(T) = \prod_{j=i}^{\rg} \frac{1}{1-T^{\delta _{j}}} = 
\prod_{j=i}^{\rg} \frac{1}{1-T^{2d_{j}}} .$$
By the first equality, $\delta _{i}$ is the smallest positive degree of the elements of 
$A_{i}$ and by the second equality, 
$2d_{i}$ is the smallest positive degree of the elements of 
$A_{i}$ too, whence $\delta _{i}=2d_{i}$. 
Then with $i=\rg$, we get that 
$\delta _{j}=2d_{j}$ for $j=1,\ldots,\rg$. 
\end{proof}

Recall that $\widehat{R}=\k[[t]]$.

\begin{coro}\label{csg1} 
Suppose that  $A_{*}$ is a polynomial algebra. Then for some standard homogeneous 
generating sequence $\poi q1{,\ldots,}{\rg}{}{}{}$ in $\ai g{}{}$,
$$ A_{*}=R_{*}[t^{- 2 d _{1}}\tau \rond \kappa (q_{1}),\ldots,
t^{- 2 d _{\rg}}\tau \rond \kappa (q_{\rg})] .$$
\end{coro}

\begin{proof}
For $m$ nonnegative integer, denote by $\ai g{}{}_{m}$ the space of standard degree $m$
of $\ai g{}{}$. By Proposition~\ref{psg1}(ii), for some homogeneous sequence 
$\poi p1{,\ldots,}{\rg}{}{}{}$ in $A_{+}$ such that $\poi p1{,\ldots,}{\rg}{}{}{}$ have 
degree $\poi {2d}1{,\ldots,}{\rg}{}{}{}$ respectively, 
$$ A_{*} = R_{*}[\poi p1{,\ldots,}{\rg}{}{}{}] .$$
For $i=1,\ldots,\rg$, let $m_{i}$ be the smallest integer such that $t^{m_{i}}p_{i}$
is in $\tau \rond \kappa (\ai g{}{})$. By Lemma~\ref{lsg1}(i), $t^{m_{i}}p_{i}$ has an 
expansion
$$ t^{m_{i}}p_{i} = \sum_{j\in {\Bbb N}} t^{j}\tau \rond \kappa (q_{i,j})$$
with $q_{i,j}$, $j\in {\Bbb N}$, in $\ai g{}{}_{d_{i}}$ of finite support. 
Denoting by $\delta _{i,j}$ the standard degree of $\ie q_{i,j}$, set:
$$
J'_{i} := \{j \in {\Bbb N} \; ; \; m_{i} = j+\delta _{i,j}\}, \qquad 
\delta _{i} := \inf \{ \delta _{i,j} \; ; \; j\in J'_{i}\},$$ 
$$j_{i} := m_{i}- 2 d_i , \qquad  
Q_{i} := t^{-2 d _{i}} \tau \rond \kappa (q_{i,j_{i}}).$$
For $i=1,\ldots,\rg$, since $p_{i}$ is not divisible by $t$ in $A$, 
$$ p_{i} - Q_{i} \in tA ,$$
whence
$$ A_{*} \subset R_{*}[\poi Q1{,\ldots,}{\rg}{}{}{}] + tA_{*} .$$
Then, by induction $m$, 
$$ A_{*} \subset R_{*}[\poi Q1{,\ldots,}{m}{}{}{}] + t^{m}A_{*} $$
for all $m$. As a result, 
$$ \widehat{A} = \widehat{R}[\poi Q1{,\ldots,}{\rg}{}{}{}] ,$$
since for all $d$, the $R_{*}$-module $A_{*}^{[d]}$ is finitely generated. Then, 
by Lemma~\ref{l2gf1}, 
$$ A_{*} = R_{*}[\poi Q1{,\ldots,}{\rg}{}{}{}] .$$
As a result, since $A$ has dimension $\rg +1$, 
the elements $t,\poi Q1{,\ldots,}{\rg}{}{}{}$ are 
algebraically independent over $\k$ and so are 
$\poi q{1,j_{1}}{,\ldots,}{\rg,j_{\rg}}{}{}{}$. Moreover
the algebra $\ai g{}{}$ is generated by $\poi q{1,j_{1}}{,\ldots,}{\rg,j_{\rg}}{}{}{}$
since they have degree $\poi d1{,\ldots,}{\rg}{}{}{}$ respectively.
\end{proof}

\subsection{} \label{sg2}
Denote by ${\cal V}$ the nullvariety of $A_{+}$ in ${\goth g}^{f}\times \k$. Let 
${\cal V}_{*}$ be the union of the irreducible components of ${\cal V}$ which are not
contained in ${\goth g}^{f}\times \{0\}$. The following result is proven in 
\cite[Corollary 4.4(i)]{CM1}. Indeed, the proof 
of this result does not use the assumption of 
\cite[Section 4]{CM1} that for some homogeneous generators $q_1,\ldots,q_\ell$  of 
${\rm S}(\g)^\g$, the elements $\ie{q_1},\ldots,\ie{q_\ell}$ are algebraically 
independent. 

\begin{lemma}[{\cite[Corollary 4.4(i)]{CM1}}]\label{lsg2}
\begin{enumerate}
\item[{\rm (i)}] The variety ${\cal V}_{*}$ is equidimensional of dimension $r+1-\rg$.
\item[{\rm (ii)}] 
For all irreducible component $X$ of ${\cal V}_{*}$ and for all $z$ in 
$\k$, $X$ is not contained in ${\goth g}^{f}\times \{z\}$.
\end{enumerate}
\end{lemma}

Let ${\cal N}$ be the nullvariety of $\varepsilon (A)_{+}$ in ${\goth g}^{f}$. Then 
${\cal V}$ is the union of ${\cal V}_{*}$ and ${\cal N}\times \{0\}$.

\begin{lemma}\label{l2sg2}
\begin{enumerate}
\item[{\rm (i)}] All irreducible component of ${\cal N}$ have dimension at least $r-\rg$ 
and all irreducible component of ${\cal V}$ have dimension at least $r+1-\rg$.
\item[{\rm (ii)}]  Assume that ${\cal N}$ has dimension $r-\rg$. Then for some 
homogeneous sequence $\poi p1{,\ldots,}{r-\rg}{}{}{}$ in $\es S{{\goth g}^{e}}_{+}$, the 
nullvariety of $t,\poi p1{,\ldots,}{r-\rg}{}{}{}$ in ${\cal V}$ is equal to $\{0\}$.
\end{enumerate}
\end{lemma}

\begin{proof}
(i) By Lemma~\ref{lsg1}(ii), for some homogeneous sequence 
$\poi a1{,\ldots,}{\rg}{}{}{}$ in $A_{+}$, the elements $t,\poi a1{,\ldots,}{\rg}{}{}{}$ 
are algebraically independent over $\k$. Let $\poi b1{,\ldots,}{m}{}{}{}$ be a 
homogeneous sequence in $A_{+}$, generating the ideal $\es S{{\goth g}^{e}}[t]A_{+}$ of 
$\es S{{\goth g}^{e}}[t]$. Set:
$$ B := \k[\poi a1{,\ldots,}{\rg }{}{}{},\poi b1{,\ldots,}{m}{}{}{}], 
\qquad B_{+} := \poi {Ba}1{+\cdots +}{\rg }{}{}{}+\poi {Bb}1{+\cdots +}{m}{}{}{},$$
$$C := B[t], \qquad C_{++} := B_{+}[t] + Ct .$$
Then $B$ and $C$ are graded subalgebras of $A$ and $B_{+}$ and $C_{++}$ are
maximal ideals of $B$ and $C$ respectively. Moreover, $C$ has dimension $\rg +1$.
We have a commutative diagram
$$\xymatrix{ & {\goth g}^{f}\times \k \ar[ld]_{\alpha } \ar[rd]^{\beta } & \\
{\mathrm {Specm}}(C) \ar[rr]^{\pi } & & {\mathrm {Specm}}(B) }$$
with $\alpha $, $\beta $, $\pi $ the morphisms whose comorphisms are the 
canonical injections
$$ C\hookrightarrow \es S{{\goth g}^{e}}[t], \quad B \hookrightarrow 
\es S{{\goth g}^{e}}[t], \quad B \hookrightarrow C$$
respectively. Since $C$ has dimension $\rg +1$, 
the irreducible components of the fibers 
of $\alpha $ have dimension at least $r-\rg$, 
whence the result for ${\cal N}$ since 
${\cal N}\times \{0\}=\alpha ^{-1}(C_{++})$. Moreover, 
${\cal V}=\beta ^{-1}(B_{+})$ and $\pi ^{-1}(B_{+})$ is a subvariety of dimension 
$1$ of ${\mathrm {Specm}}(C)$. Hence all irreducible component of ${\cal V}$ has 
dimension at least $r+1-\rg$.

(ii) Prove by induction on $i$ that there exists a homogeneous sequence 
$\poi p1{,\ldots,}{i}{}{}{}$ in $\es S{{\goth g}^{e}}_{+}$ such that the minimal prime 
ideals of $\es S{{\goth g}^{e}}$ containing $\varepsilon (A)_{+}$ and 
$\poi p1{,\ldots,}{i}{}{}{}$ have height $\rg +i$. First of all, 
$\es S{{\goth g}^{e}}\varepsilon (A)_{+}$ is graded. 
Then the minimal prime ideals of $\es S{{\goth g}^{e}}$ containing $\varepsilon (A)_{+}$ 
are graded too. By, (i), they have height $\rg$ since 
${\cal N}$ has dimension $r-\rg$ by hypothesis. In particular, they are strictly 
contained in $\es S{{\goth g}^{e}}_{+}$. Hence, by Lemma~\ref{lgf1}(ii), for some 
homogeneous element $p_{1}$ in $\es S{{\goth g}^{e}}$, $p_{1}$ is not in the union of 
these ideals so that the statement is true for $i=1$ by~\cite[Ch.~5, Theorem~13.5]{Mat}. 
Suppose that it is true for $i-1$. Then the minimal prime ideals containing 
$\varepsilon (A)_{+}$ and $\poi p1{,\ldots,}{i-1}{}{}{}$ are graded and strictly 
contained in $\es S{{\goth g}^{e}}_{+}$ by the induction hypothesis. So, by 
Lemma~\ref{lgf1}(ii), for some homogeneous element $p_{i}$ in $\es S{{\goth g}^{e}}$, 
$p_{i}$ is not in the union of these ideals and the sequence $\poi p1{,\ldots,}{i}{}{}{}$
satisfy the condition of the statement by~\cite[Ch. 5, Theorem 13.5]{Mat}. For $i=r-\rg$,
the nullvariety of $\poi p1{,\ldots,}{r-\rg}{}{}{}$ in ${\cal N}$ has dimension 
$0$. Then it is equal to $\{0\}$ as the nullvariety of a graded ideal, whence the
assertion since ${\cal N}\times \{0\}$ is the nullvariety of $t$ in ${\cal V}$.
\end{proof}

\subsection{} \label{sg3}  
We assume in this subsection that ${\cal N}$ has dimension $r-\ell$. 
Let $\poi p1{,\ldots,}{r-\rg}{}{}{}$ be as in Lemma~\ref{l2sg2}(ii), and set 
$$ C := A[\poi p1{,\ldots,}{r-\rg}{}{}{}] .$$
Then $\poi p1{,\ldots,}{r-\rg}{}{}{}$ are algebraically independent over $A$ since 
${\cal N}$ has dimension $r-\ell$. 

\begin{lemma}\label{lsg3}
The ideal $\es S{{\goth g}^{e}}[t]_{+}$ of $\es S{{\goth g}^{e}}[t]$  is the radical of 
$\es S{{\goth g}^{e}}[t]C_{+}$.
\end{lemma}

\begin{proof} 
Let $Y$ be an irreducible component of the nullvariety of $C_{+}$ in 
${\goth g}^{f}\times \k$. Then $Y$ has dimension at least $1$. By definition the 
nullvariety of $t$ in $Y$ is equal to $\{0\}$. Hence $Y$ has dimension $1$. The grading
on $\es S{{\goth g}^{e}}[t]$ induces an action of 
the one-dimensional multiplicative 
group $\gm$ on ${\goth g}^{f}\times \k$ such 
that for all $(x,z)$ in ${\goth g}^{f}\times \k$, $(0,z)$ is in the closure of the orbit 
of $(x,z)$ under $\gm$. Since $C_{+}$ is graded, $Y$ is invariant under $\gm$. As a 
result, $Y=\{0\}\times \k$ or for some $x$ in ${\goth g}^{f}\times \k$, $Y$ is the 
closure of the orbit of $(x,0)$ under $\gm$ since $0$ is the nullvariety of $t$ in $Y$. 
In the last case, $x$ is a zero of $\poi p1{,\ldots,}{r-\rg}{}{}{}$ in ${\cal N}$, that 
is $x=0$. Hence $Y=\{0\}\times \k$. As a result, the nullvariety of $C_{+}$ in 
${\goth g}^{f}\times \k$ is equal to $\{0\}\times \k$ that is the nullvariety of 
$\es S{{\goth g}^{e}}[t]_{+}$, whence the assertion since $\es S{{\goth g}^{e}}[t]_{+}$ 
is a prime ideal of $\es S{{\goth g}^{e}}[t]$.
\end{proof}

For ${\goth p}$ a prime ideal of $A$, denote by $A_{{\goth p}}$ the localization of 
$A$ at ${\goth p}$ and by $\overline{{\goth p}}$ the ideal of $C$ generated by 
${\goth p}$. Since $C$ is a polynomial algebra over $A$, $\overline{{\goth p}}$ 
is a prime ideal of $C$ and $A\setminus {\goth p}$ is the intersection of 
$A$ and $C\setminus \overline{{\goth p}}$. Hence the localization 
$C_{\overline{{\goth p}}}$ of $C$ at $\overline{{\goth p}}$ is a localization of the 
polynomial algebra $A_{{\goth p}}[\poi p1{,\ldots,}{r-\rg}{}{}{}]$. Moreover, 
$A_{{\goth p}}$ is the quotient of $C_{\overline{{\goth p}}}$ by the ideal generated 
by $\poi p1{,\ldots,}{r-\rg}{}{}{}$. According to \cite[Ch.~6, Theorem~17.4]{Mat}, if 
$C_{\overline{{\goth p}}}$ is Cohen-Macaulay, $\poi p1{,\ldots,}{r-\rg}{}{}{}$ is a 
regular sequence in $C_{\overline{{\goth p}}}$ since $A_{{\goth p}}$ has dimension
$\dim C_{\overline{{\goth p}}}-r+\rg$. Then, again by \cite[Ch.~6, Theorem~17.4]{Mat},
$A_{{\goth p}}$ is Cohen-Macaulay if so is $C_{\overline{{\goth p}}}$.  

\begin{proof}[Proof of Theorem \ref{ti6}]
By Lemma~\ref{lsg3} and Proposition~\ref{pgf1}, the algebra $C$ is finitely generated. 
Then $A$ is finitely generated as a quotient of $C$. 
Hence by Lemma~\ref{lgf2}(ii), $A$ is
a factorial ring and so is $C$ as a polynomial ring over $A$. As a result, $C$ is normal
so that $\es S{{\goth g}^{e}}[t]$ and $C$ satisfy the conditions (1), (2), (3) of 
Proposition~\ref{p2gf1}. Hence by Proposition~\ref{p2gf1}, for all prime ideal 
${\goth p}$ of $A$, containing $t$, $C_{\overline{{\goth p}}}$ is Cohen-Macaulay, whence
$A_{{\goth p}}$ is Cohen-Macaulay. By Lemma~\ref{lsg1}(ii), for ${\goth p}$ a prime 
ideal of $A$, not containing $t$, $A_{{\goth p}}$ is the localization of 
$\k[t,t^{-1},\poi a1{,\ldots,}{\rg}{}{}{}]$ at the prime ideal generated by ${\goth p}$.
Therefore $A_{{\goth p}}$ is Cohen-Macaulay since the algebra 
$\k[t,t^{-1},\poi a1{,\ldots,}{\rg}{}{}{}]$ is regular. As a result $A$ is Cohen-Macaulay.
In particular, $A$ satisfies the conditions (1), (2), (3) of Subsection~\ref{gf2}. So, by 
Proposition~\ref{pgf2}, $A_{*}$ is a polynomial algebra 
over $R_*$. Then by Corollary~\ref{csg1}, 
for some homogeneous generating sequence $\poi q1{,\ldots,}{\rg}{}{}{}$ in 
$\ai g{}{}$, 
$$ A_{*} = R_{*}[t^{- 2 d _{1}}\tau \rond \kappa (q_{1}),\ldots,
t^{- 2 d_{\rg}}\tau \rond \kappa (q_{\rg})].$$
Form the above equality, we deduce that any element of $A$ is the product of an element 
of the algebra 
$R[t^{- 2 d _{1}}\tau \rond \kappa (q_{1}),\ldots,
t^{- 2 d_{\rg}}\tau \rond \kappa (q_{\rg})]$ 
by a polynomial in $t$ with nonzero constant term, whence 
$$ A= R[t^{- 2 d _{1}}\tau \rond \kappa (q_{1}),\ldots,
t^{- 2 d_{\rg}}\tau \rond \kappa (q_{\rg})] 
\quad \text{ and so }\quad  
\varepsilon (A) = \k[\poi {\ie q}1{,\ldots,}{\rg}{}{}{}]$$   
since for $i=1,\ldots,\rg$, 
$$\ie q_{i}:=\varepsilon (t^{- 2 d _{i}}\tau \rond \kappa (q_{i})) .$$ 
Since ${\cal N}\times \{0\}$ is the nullvariety of $t$ and $A_{+}$ in 
${\goth g}^{f}\times \k$, ${\cal N}$ is the nullvariety in ${\goth g}^{f}$ of 
$\poi {\ie q}1{,\ldots,}{\rg}{}{}{}$. Hence $\poi {\ie q}1{,\ldots,}{\rg}{}{}{}$ are 
algebraically independent over $\k$ since ${\cal N}$ has dimension $r-\rg$.
\end{proof}

\section{Proof of Theorem \ref{ti5}} \label{proof}
Let $(e,h,f)$ be an ${\goth {sl}}_{2}$-triple in ${\goth g}$. 
We use the notations $\kappa $ and $\ie p$, $p\in \ai g{}{}$, 
as in the introduction. 
In this section, we use the standard gradings on $\e Sg$ and $\es S{{\goth g}^{e}}$. 
Let $A_0$ be the subalgebra of $\es S{{\goth g}^{e}}$ generated by the family
$\ie p, \; p \in \ai g{}{}$, and let ${\cal N}_0$ be the nullvariety of $A_{0,+}$ in 
${\goth g}^{f}$ where $A_{0,+}$ denotes the augmentation ideal of $A_0$. 

Let $\poi a1{,\ldots,}{m}{}{}{}$ be a homogeneous sequence in $A_{0,+}$ generating the
ideal of $\es S{{\goth g}^{e}}$ generated by $A_{0,+}$. According 
to~\cite[Corollary 2.3]{PPY}, $A_0$ contains homogeneous elements 
$\poi b1{,\ldots,}{\rg}{}{}{}$ algebraically independent over $\k$. 

\begin{lemma}\label{lproof}
Let ${\goth A}$ be the integral closure of 
$\k[\poi a1{,\ldots,}{m}{}{}{},\poi b1{,\ldots,}{\rg}{}{}{}]$ in the fraction field of 
$\es S{{\goth g}^{e}}$.

\begin{enumerate}
\item[{\rm (i)}] The algebra ${\goth A}$ is contained in $\ai ge{}$ and its fraction field
is the fraction field of $\ai ge{}$.
\item[{\rm (ii)}] Let $a$ in $\ai ge{}_{+}$. If $a$ is equal to $0$ on ${\cal N}_0$, then
$a$ is in ${\goth A}_{+}$.
\item[{\rm (iii)}] The algebra ${\goth A}$ is the integral closure of $A_0$ in 
the fraction field of $\es S{{\goth g}^{e}}$.
\end{enumerate}
\end{lemma}

\begin{proof}
(i) Let $K_{0}$ be the field of invariant elements under the adjoint action of 
${\goth g}^{e}$ in the fraction field of $\es S{{\goth g}^{e}}$. According to
\cite[Lemma 3.1]{CM1}, $K_{0}$ is the fraction field of $\ai ge{}$. Since 
$\poi a1{,\ldots,}{m}{}{}{},\poi b1{,\ldots,}{\rg}{}{}{}$ are in $\ai ge{}$, 
${\goth A}$ is contained in $K_{0}$. Moreover, ${\goth A}$ is contained in $\ai ge{}$ 
since $\ai ge{}$ is integrally closed in $K_{0}$. Since $K_{0}$ has transcendence  
degree $\rg$ over $\k$ and since $\poi b1{,\ldots,}{\rg}{}{}{}$ are algebraically 
independent over $\k$, $K_{0}$ is the fraction field of ${\goth A}$.

(ii) Since ${\cal N}_0$ is the nullvariety of 
$\poi a1{,\ldots,}{m}{}{}{},\poi b1{,\ldots,}{\rg}{}{}{}$ in ${\goth g}^{f}$, 
${\cal N}_0$ is the nullvariety of ${\goth A}_{+}$ in ${\goth g}^{f}$. Let $a$ be in 
$\ai ge{}_{+}$ such that $a$ is equal to $0$ on ${\cal N}_0$. Since ${\cal N}_0$
is a cone, all homgogeneous components of $a$ is equal to $0$ on ${\cal N}_0$. So it
suffices to prove the assertion for $a$ homogeneous. 
We have a commutative diagram
$$ \xymatrix{ {\goth g}^{f} \ar[rr]^{\pi } \ar[rd]_{\alpha } && 
{\mathrm {Specm}}({\goth A}[a]) \ar[ld]^{\beta } \\ & {\mathrm {Specm}}({\goth A}) & }$$
with $\pi $, $\alpha $, $\beta $ the comorphisms of the canonical injections
$$ {\goth A}[a] \hookrightarrow \es S{{\goth g}^{e}}, \quad
{\goth A} \hookrightarrow \es S{{\goth g}^{e}}, \quad 
{\goth A} \hookrightarrow {\goth A}[a] .$$
Since ${\cal N}_0$ is the nullvariety of ${\goth A}[a]_{+}$ and ${\goth A}_{+}$ in 
${\goth g}^{f}$, $\beta ^{-1}({\goth A}_{+})={\goth A}[a]_{+}$. The gradings of 
${\goth A}$ and ${\goth A}[a]$ induce actions of $\gm$ on 
${\mathrm {Specm}}({\goth A})$ 
and ${\mathrm {Specm}}({\goth A}[a])$ such that $\beta $ is equivariant. Moreover, 
${\goth A}_{+}$ is in the closure of all orbit under $\gm$ in 
${\mathrm {Specm}}({\goth A})$. Hence $\beta $ is a quasi finite morphism. 
Moreover, $\beta $ is a birational since ${\goth A}$ and ${\goth A}[a]$ have the same 
fraction field by (i). 
Hence, by Zariski's main theorem \cite{Mu}, $\beta $ is an open immersion 
from ${\mathrm {Specm}}({\goth A}[a])$ into ${\mathrm {Specm}}({\goth A})$. So, 
$\beta $ is surjective since ${\goth A}_{+}$ is in the image of $\beta $ and since it is 
in the closure of all $\gm$-orbit in ${\mathrm {Specm}}({\goth A})$. As a 
result, $\beta $ is an isomorphism and $a$ is in ${\goth A}$, whence the assertion.

(iii) By (ii), $A_0$ is contained in ${\goth A}$. Moreover, since 
$\poi a1{,\ldots,}{m}{}{}{},\poi b1{,\ldots,}{\rg}{}{}{}$ are in $A_0$, 
${\goth A}$ is contained in the integral closure of $A_0$ in the 
fraction field of $\es S{{\goth g}^{e}}$, whence the assertion.
\end{proof}

\begin{coro}\label{cproof}
Suppose that the algebra $\ai ge{}$ is finitely generated. Then ${\goth A}$ is
equal to $\ai ge{}$.
\end{coro}

\begin{proof}
Let $C$ be the quotient of $\ai ge{}$ by the ideal $\ai ge{}{\goth A}_{+}$. 
By hypothesis, $C$ is finitely generated. Then it has finitely many minimal prime
ideals. Denote them by $\poi {{\goth p}}1{,\ldots,}{m}{}{}{}$. For 
$a$ in the radical of $\ai ge{}{\goth A}_{+}$, $a$ is equal to $0$ on ${\cal N}_0$. 
Moreover, it is in $\ai ge{}_{+}$. Then, by Lemma~\ref{lproof}(ii), $a$ is in 
${\goth A}_{+}$. As a result, $C$ is a reduced algebra and the canonical map
$$ C \longrightarrow \poi {C/{\goth p}}1{\times \cdots \times }{m}{}{}{}$$
is injective. Since ${\goth A}$ and $\ai ge{}$ have the same fraction field, they have 
the same Krull dimension. Denote by $d$ this dimension 
and by ${\goth p}'_{j}$, for $j=1,\ldots,m$, the 
inverse image of ${\goth p}_{j}$ in $\ai ge{}$ by the quotient map 
$\ai ge{} \to C$. 

\begin{claim}\label{clproof}
Let $j=1,\ldots,m$. For $i=1,\ldots,d$, there exists a sequence 
$\poi c1{,\ldots,}{i}{}{}{}$ of elements of ${\goth A}_{+}$ and an increasing sequence
$$ \{0\} = \poi {{\goth q}}0{ \subsetneq \cdots \subsetneq }{i}{}{}{} 
\subset {\goth p}'_{j}$$ 
of prime ideals of of $\ai ge{}$ such that $c_{i}$ is not ${\goth q}_{i-1}$ and 
$\poi c1{,\ldots,}{j}{}{}{}$ are in ${\goth q}_{j}$ for $j=1,\ldots,i$.
\end{claim}

\begin{proof}[Proof of Claim~\ref{clproof}] 
Prove the claim by induction on $i$. Let $c_{1}$ be in ${\goth A}_{+}\setminus \{0\}$. 
As ${\goth A}_{+}$ is contained in ${\goth p}'_{j}$, there exists a minimal prime ideal 
${\goth q}_{1}$ of $\ai ge{}$, contained in ${\goth p}'_{j}$ and containing $c_{1}$. 
Suppose $i>1$ and the claim true for $i-1$. As the sequence 
$$ \{0\} = \poi {{\goth A}_{+}\cap {\goth q}}1{\subsetneq \cdots \subsetneq}{i-1}{}{}{}
\subset {\goth A}_{+}$$
is an increasing sequence of prime ideals of ${\goth A}_{+}$ and ${\goth A}_{+}$ has 
height $d>i-1$, ${\goth A}_{+}$ is not contained in ${\goth q}_{i-1}$. Let $c_{i}$ be
in ${\goth A}_{+}\setminus {\goth q}_{i-1}$ and ${\goth q}_{i}$ the minimal prime 
ideal of $\ai ge{}$ contained in ${\goth p}'_{j}$ and containing $c_{i}$ and 
${\goth q}_{i-1}$. So by the induction hypothesis, the sequence 
$c_1,\ldots,c_i$ satisfies the conditions of the claim. This concludes the proof. 
\end{proof}

By the claim, ${\goth p}'_{j}$ has height at least $d$ for $j=1,\ldots,m$. 
Hence 
$\poi {{\goth p}}1{,\ldots,}{m}{}{}{}$ are maximal ideals of $C$. As a result, 
the $\k$-algebra $C$ is finite dimensional. Let $V$ be a graded complement to
$\ai ge{}{\goth A}_{+}$ in $\es S{{\goth g}^{e}}$. 
From the equality ${\rm S}(\g^{e}) = V + {\rm S}(\g^{e})^{\g^{e}} {\goth A}_{+}$, 
we get that ${\rm S}(\g^{e}) = V{\goth A} + 
{\rm S}(\g^{e})^{\g^{e}}{\goth A}_{+} ^{m}$ for any nonnegative 
integer $m$ by induction on $m$. 
Hence $\ai ge{}=V{\goth A}$ so that $\ai ge{}$ is a finite extension of ${\goth A}$. 
Since ${\goth A}$ is integrally closed in the fraction field of $\ai ge{}$, 
${\goth A}=\ai ge{}$.  
\end{proof}

\begin{proof}[Proof of Theorem \ref{ti5}]
The ``if'' part results from \cite[Theorem 1.5]{CM1} (or, here, Theorem \ref{ti4}). 

Suppose that $e$ is good. By Definition \ref{di} and Theorem \ref{ti3}, 
$\ai ge{}$ is a polynomial algebra and the nullvariety of $\ai ge{}_{+}$ in 
${\goth g}^{f}$ is equidimensional of dimension $r-\rg$. On the other hand, 
by Lemma~\ref{lproof}(iii), $\mathfrak{A}$ is the integral closure of $A_0$ in the 
fraction field of ${\rm S}(\g^{e})$. Hence the nullvarieties of $\mathfrak{A}_+$ and 
$A_{0,+}$ in $\g^{f}$ are the same.  But by~Corollary~\ref{cproof}, 
$\mathfrak{A}=\ai ge{}$, so ${\cal N}_0$ has dimension $r-\rg$ since $e$ is good. 
On the other hand, $A_0$ is contained in $\varepsilon(A)$ by construction of 
$\varepsilon(A)$, and $\varepsilon(A)$ is contained in ${\rm S}(\g^{e})^{\g^{e}}$ by 
\cite[Proposition 0.1]{PPY}, whence ${\cal N}= {\cal N} _0$.

As a result, ${\cal N}$ has dimension $r-\ell$ and so by Theorem~\ref{ti6}, for some 
homogeneous generating sequence $\poi q1{,\ldots,}{\rg}{}{}{}$ of $\ai g{}{}$, the 
element $\poi {\ie q}1{,\ldots,}{\rg}{}{}{}$ are algebraically independent over $\k$.
\end{proof}


\begin{thebibliography}{DK00}
\bibitem[B98]{Bou} N.~Bourbaki, 
{\em Alg\`ebre commutative, Chapitre 10, \'El\'ements de math\'ematiques}, 
Masson (1998), Paris.

\bibitem[CM10]{CM} J.-Y.~Charbonnel and A.~Moreau, 
{\em The index of centralizers of elements of reductive 
Lie algebras}, Documenta Mathematica, {\bf 15} (2010), 387-421. 

\bibitem[CM16]{CM1} J.-Y.~Charbonnel and A.~Moreau, 
{\em The symmetric invariants of centralizers and Slodowy grading}, 
Math. Zeitschrift {\bf 282} (2016), {\bf n$^{\circ}$1-2}, 273-339. 

\bibitem[JS10]{JS} A.~Joseph and D.~Shafrir, 
{\em Polynomiality of invariants, unimodularity and adapted pairs}, Transformation 
Groups, {\bf 15}, (2010), 851--882.

\bibitem[Ma86]{Mat} H.~Matsumura, 
{\em Commutative ring theory} 
Cambridge studies in advanced mathematics (1986), {\bf n$^{\circ}$8}, Cambridge
University Press, Cambridge, London, New York, New Rochelle, Melbourne,
Sydney.

\bibitem[Mu88]{Mu} D.~Mumford, 
{\em The Red Book of Varieties and Schemes}, Lecture Notes
in Mathematics (1988), {\bf n$^{\circ}$ 1358}, Springer-Verlag, Berlin,
Heidelberg, New York, London, Paris, Tokyo.

\bibitem[PPY07]{PPY} D.I.~Panyushev, A.~Premet and O.~Yakimova, 
{\em On symmetric invariants of centralizers in reductive Lie algebras}, 
Journal of Algebra {\bf 313} (2007), 343--391.

\bibitem[Y07]{Y3} O.~Yakimova, 
{\em A counterexample to Premet's and Joseph's conjecture},
Bulletin of the  London Mathematical Society {\bf 39} (2007), 749--754.

\bibitem[Y16]{Y4} O.~Yakimova, 
{\em Symmetric invariants of $\mathbb{Z}_2$-contractions and other semi-direct products}, 
preprint 2016. 
\end{thebibliography}
\end{document}